\documentclass[12pt,reqno]{amsart}
\usepackage[all]{xy}
\usepackage{amsfonts,amssymb,amsmath}
\usepackage{enumitem}
\usepackage{datetime}
\usepackage{tikz}
\usepackage{relsize}
\usepackage{mathrsfs} 
 
\usepackage{MnSymbol} 

\makeatletter

\renewcommand\subsubsection{\@secnumfont}{\bfseries}%
\renewcommand\subsubsection{\@startsection{subsubsection}{3}
	\z@{.5\linespacing\@plus.7\linespacing}{-.5em}%
	{\normalfont\bfseries}}

\makeatother

 \usepackage[colorlinks=true, citecolor=cyan, urlcolor=black, pagebackref,
 linkcolor=blue ]{hyperref}

\usepackage{xcolor}

\definecolor{fgreen}{RGB}{44,144, 14}

\renewenvironment{proof}{{\bfseries Proof.}}{\qed}

\numberwithin{equation}{section} 
 
\newtheorem{theorem}{Theorem}[section] 
\newtheorem{proposition}[theorem]{Proposition} 
\newtheorem{corollary}[theorem]{Corollary} 
\newtheorem{lemma}[theorem]{Lemma} 

\theoremstyle{definition}
\newtheorem{definition}[theorem]{Definition} 
\newtheorem{remark}[theorem]{Remark} 
\newtheorem{example}[theorem]{Example} 

\def\R{\mathbb {R}}
\def\C{\mathbb {C}}
\def\N{\mathbb {N}}
\def\H{\mathbb {H}}
\def\Z{\mathbb {Z}}

\def\SS{\mathbb {S}}

\def\Ab{\mathbf {A}}

\def\Sb{\mathbf{S}}

\def\CC{\mathcal {C}}

\def\CS{\mathscr {C}}

\def\NC{\mathcal {N}}

\def\BC{\mathcal {B}}

\def\IC{\mathcal {I}}

\def\g{\mathfrak {g}}
\def\h{\mathfrak {h}}

\def\k{\mathfrak {k}}
\def\s{\mathfrak {s}}
\def\a{\mathfrak {a}}

\def\u{\mathfrak {u}}
\def\p{\mathfrak {p}}

\def\z{\mathfrak {z}}
\def\c{\mathfrak {c}}

\def\b{\mathfrak b}
\def\l{\mathfrak l}
\def\o{\mathfrak o}
\def\e{\mathfrak e}

\def\<>{\langle \cdot ,\, \cdot \rangle}
\def\Bi{\rm B} 
\def\Re{{\rm R}_{\g, \h}}

\def\Im{{\rm Im\,}}
\def\ker{{\rm ker\,}}
\def\ox{\otimes}

\def\cds{\cdots}
\def\th{\theta}
\def\dl{\delta}

\def\lto{\longrightarrow}
\def\nb{\nabla}
\def\n2{\nabla^2}
\def\n3{\nabla^3}

\begin{document} 
 
\title[Third and fourth Betti numbers of homogeneous spaces]{On the third and fourth Betti numbers  of a homogeneous space of a Lie group}

\author[I. Biswas]{Indranil Biswas} 

\address{Department of Mathematics, Shiv Nadar University, NH91, Tehsil Dadri, Greater Noida 201314, Uttar Pradesh, India}

\email{indranil.biswas@snu.edu.in, indranil29@gmail.com}

\author[P. Chatterjee]{Pralay Chatterjee}

\address{The Institute of Mathematical Sciences, HBNI, CIT Campus, 	Tharamani, Chennai 600113, India}

\email{pralay@imsc.res.in} 

\author[C. Maity]{Chandan Maity}

\address{Indian Institute of Science Education and Research (IISER) Berhampur,   Berhampur 760003, Odisha, India}

\email{cmaity@iiserbpr.ac.in} 



\begin{abstract}
In \cite{BCM} explicit descriptions of the second and first real de Rham cohomology groups of a general homogeneous space of a Lie group are given, extending an earlier result in \cite{BC}. From the computational viewpoint they turned out to be new and very useful, and in fact
played a crucial role in determining the second cohomology of nilpotent orbits as in \cite{BC} and \cite{BCM}. In this paper, we give computable and explicit descriptions of the third and fourth
real de Rham cohomologies of a general homogeneous space, in terms of the associated Lie-theoretic data, along the lines mentioned above. We also draw numerous corollaries of our main results in important special settings. Moreover, as a consequence, we obtain a new and interesting invariant by showing that for a large class of homogeneous spaces, the difference between
the third and fourth Betti numbers coincides with the difference between the numbers of simple factors of the ambient group and the associated closed subgroup.
\end{abstract}

\maketitle

\tableofcontents

\definecolor{theoremcolor}{RGB}{0,0,128} 
\definecolor{defcolor}{RGB}{0,100,0} 
\definecolor{corcolor}{RGB}{139,0,0} 

\newcommand{\highlightthm}[1]{\textcolor{theoremcolor}{\textbf{#1}}}
\newcommand{\highlightdef}[1]{\textcolor{defcolor}{\textbf{#1}}}
\newcommand{\highlightcor}[1]{\textcolor{corcolor}{\textbf{#1}}}

\section{Introduction} 

\subsection{Context and motivation}

The study of the  de Rham cohomology $H^*(G/H)$ of a homogeneous space $G/H$, where $G$ is a connected Lie group and $H$ is a closed subgroup, in relation to the algebraic invariants of $G$, $H$, and the inclusion $H \subset G$, constitutes a classical theme in differential geometry and topology. 
Due to a minor variation of a result of Mostow (see Theorem \ref{mostow}), one finds that when $H$ has finitely many connected components, the problem of determining topological invariants  such as the cohomology of the homogeneous space $G/H$, reduces to the case where the Lie group $G$ is compact.  
This problem has been the focus of extensive investigation by many mathematicians in the past; see \cite{Ca2,Bo} for the famous works of Cartan and Borel, respectively.  
As above, if $G$ is now a
compact connected Lie group with $H\,\subset\, G$ a closed
subgroup, which is not necessarily connected, and if $\g$ and $\h$ are the respective Lie algebras, 
then it follows that the cohomology of $G/H$ is the 
relative Lie algebra cohomology of the chain complex 
of forms on $\g$ which vanish on $\h$ and which are invariant 
under the natural action of $H$. When $H$ is connected, 
the Cartan's Theorem (see \cite[Theorem 5]{Ca2}, 
\cite[Theorem 25.2]{Bo}) gives an alternative description of 
the cohomology of the above cochain complex in terms of 
the cohomology of a certain associated Koszul complex (defined in Section \ref{sec-Koszul-complex}).
See \cite{Bo}, \cite{Ba}, \cite{Wo}, and, more recently, \cite{Fr} for 
other viewpoints on Cartan's theorem, many general results, and 
applications in various special settings. We also 
refer to \cite{GS} and \cite{Tu} for further expositions on  Cartan's 
theorem and its ramifications.

While the above results provide a foundational basis,
they do not, in general,  seem to provide
effective computational methods for determining the cohomology or other related topological invariants of homogeneous spaces in broad generality.
In the context of computable descriptions of  such invariants, 
 the only major general results that we are aware of are due  
 to A. Borel, H. Hopf, H. Samelson and H-C. Wang. 
In \cite{HS} Hopf and Samelson have shown that the Euler characteristic of a homogeneous space
vanishes if and only if the rank of the associated ambient compact group is strictly bigger than that of the closed subgroup, and   
Wang proved in \cite{Wa} that if the ranks of the ambient group and the closed subgroup are equal, then the Euler characteristic equals the quotient of the orders of their respective Weyl groups.  In \cite{Bo} Borel gave an explicit description of the Poincar\'{e} polynomial of a homogeneous space when  the associated ambient compact Lie group and the closed subgroup are of equal rank.   
Although the Poincar\'{e} polynomial produces all the Betti numbers, it should also be noted that the
equal-rank condition imposes a significant restriction.

Surprisingly enough, apart from the results in \cite{Bo, HS, Wa} mentioned above, 
which are proved in the past, there seems to be hardly 
any literature describing cohomologies or topological invariants in explicit ways, despite the availability of the above striking Cartan's theorem. Thus from the 
computational point of view, it would be desirable to obtain simple descriptions of the cohomology groups, which should be based entirely on Lie theoretic data such as the Lie algebras of the associated groups, how the  closed subgroup or the corresponding subalgebra is positioned in the ambient Lie group or the Lie algebra, and the action of the component group of the 
closed subgroup on relevant spaces. As beginning steps towards addressing these questions,
in the recent past in \cite[Theorem 3.3, Theorem 3.6]{BCM}  simple and computable descriptions of the second and first cohomologies of general compact connected homogeneous spaces are obtained, extending an earlier result in  \cite[Theorem 3.3]{BC}.  
Such descriptions turned out to be extremely useful in \cite{BC, BCM}, where second cohomology groups of nilpotent orbits are computed for classical Lie algebras. See also \cite{CM} for some more applications on the second cohomology of nilpotent orbits in exceptional Lie algebras
along the same line.  
In the same spirit as our above work, in this paper we derive explicit formulas for the  third and fourth de Rham cohomology groups of a compact connected homogeneous space
in the most general setting; see Theorem \ref{thm-3rd-coh-general} and Theorem \ref{thm-4th-coh-general} appearing below.
We also refer to \cite{A} and the recent work \cite{LW} for results on third Betti numbers when the ambient group is semisimple and the associated closed subgroup is connected; 
see also Section \ref{comparison} and Remark \ref{remark-LW} on 
chronological comparison and mathematical relationships with our results.
As applications of our results, we also  obtain a number of  corollaries in important special cases; see Section \ref{appl-thm}.   Furthermore, 
as an interesting noteworthy  consequence of the above corollaries, we obtain a 
new and elegant description of a topological invariant in the setting  where the ambient compact group  is semisimple; see Theorem   \ref{thm-invariant},   and Corollary \ref{cor-homotopy-inv}.

\subsection{Description of the  results}
We now describe our main results and some of
their corollaries. Throughout this paper by a {\it semisimple (respectively, simple) Lie group} we mean a {\it connected non-trivial Lie group with semisimple 
(respectively, simple) Lie algebra}. 

\subsubsection{The main results}
We begin by fixing some basic notions for immediate and later use. We also refer to   Section \ref{notation} for other relevant notation.

In what follows a
Lie group will be denoted by a capital letter, and the associated Lie algebra will be denoted by the 
corresponding lower case German letter, unless a different notation is explicitly mentioned.

Let $\Gamma$ be a group acting linearly on a vector space $V$. The {\it subspace of vectors in $V$ fixed pointwise by the action of $\Gamma$} is denoted by $V^\Gamma$. 
If $\Gamma_1 \subset \Gamma$ is a normal subgroup acting trivially on
$V$ then $\Gamma/ \Gamma_1$ has a natural action on $V$,
and moreover, $V^\Gamma \,=\, V^{\Gamma / \Gamma_1}$.
We now define a  space that is naturally  associated to a pair of a Lie group $L$, not necessarily connected, and a Lie subalgebra ${\mathfrak r} \subset {\mathfrak l}$. Assuming that $\mathfrak r$ remains invariant under the adjoint action of $L$,
the {\it{ space of symmetric $L$-invariant bilinear  maps}} on $\mathfrak r$  is  denoted by $S^2(\mathfrak r^*)^L$. 

\begin{definition}  
	\label{def-NC-gH} 
	Let $G$ be a compact connected Lie group, and $ H \,\subseteq\, G $ be a closed subgroup which is not necessarily connected.  
	Let $\Psi_{\g, \h} \,:\, S^2( [\g,\,\g]^*)^G \,\lto\, S^2( (\h\,\cap\, [\g,\,\g])^*)^H$
	be the restriction map defined by
	\begin{align}\label{map-Psi-g-h}
		\Psi_{\g, \h} ( \phi  )\,:=\,    \phi|_{  (\h\,\cap\, [\g,\,\g])\, \times \,  (\h\,\cap\, [\g,\,\g])}\,.
	\end{align}
	We then set
	\hspace{2cm}
	$
	\NC_{\g,\h}:= \ker \Psi_{\g, \h},
	\text{ and } \     \CC_{\g,H}
	:= \dfrac{S^2( (\h\,\cap\, [\g,\,\g])^*)^H}{ {\rm Im}\,\Psi_{\g, \h}}
	= {\rm coker}\, \Psi_{\g, \h}\,.  
	$	  \qed
\end{definition}

The following main results describe the third and fourth cohomologies of a connected compact homogeneous space in the most general setting.

\begin{theorem}\label{thm-3rd-coh-general}
 Let $G$ be a compact connected Lie group, and $H$ be a closed  subgroup of $G$ which is not necessarily connected. 
Then there is an isomorphism: 
\begin{align*}
H^3(G/H )\, \simeq\,\big(({\z}(\h) \cap [\g,\,\g])^*\big)^{H/H^0} \ox {\Big(\frac{\g}{[\g,\,\g]+\h}\Big)}^*
\,\bigoplus \, \NC_{\g,\h}\, \bigoplus\, \wedge^3 {\Big(\frac{\g}{[\g,\,\g]+\h}\Big)}^*\,.
\end{align*}
Moreover, 
\begin{enumerate}

\item If $\z(\h) \cap [\g,\,\g] =0$ then 
$H^3(G/H)\simeq H^3(G/H^0)$. 
\item If $\z(\g) \subset \h$ (in particular, if $G$ is semisimple) then
$H^3(G/H)\simeq H^3(G/H^0) \simeq  \NC_{\g, \h}\,. $
\end{enumerate}
\end{theorem}

\begin{theorem}\label{thm-4th-coh-general}
Let $G$ be a compact connected Lie group, and  $H$ be a closed  subgroup of $G$ which is not necessarily connected. 	Let $r_0\,:=\,\dim  \big(\frac{\g}{[\g,\,\g]+\h}\big)$. Then there is an isomorphism: 
\begin{align*}
H^4(G/H )\,\ \simeq\,\ &
		\big(({\z}(\h) \cap [\g,\,\g])^*\big)^{H/H^0} \ox\wedge^2 {\Big(\frac{\g}{[\g,\,\g]+\h}\Big)}^* \,
\bigoplus\, \,
	(\NC_{\g,\h})^{r_0}\\
&\,\,\bigoplus\,\,\CC_{\g,H}\,\,\bigoplus\,\,\wedge^4 {\Big(\frac{\g}{[\g,\,\g]+\h}\Big)}^*\, .
	\end{align*}
\begin{enumerate}

	\item If $\z(\g) \subset \h$ (in particular, if $G$ is semisimple) then $H^4(G/H )\, \simeq\,\CC_{\g,H}\,.$
	\item  If $G$ and $H^0 $ are semisimple Lie groups and all the simple factors of $H^0$ are pairwise non-isomorphic, then
	$ H^4(G/H ) \,\simeq\, H^4(G/H^0  )$.
\end{enumerate}
	\end{theorem} 

\medskip 

\begin{remark}
Observe that when $G$ is semisimple, $H^3(G/H)$ does not depend on the connectedness of $H$, but
$H^4(G/H)$ can depend on the connectedness of $H$.  \qed
\end{remark}

In a broad sense  the notions $ \NC_{\g,\h} $ and $\CC_{\g,H}$ (see Definition \ref{def-NC-gH}) depend directly on Lie algebra $\g$, how $H$ (respectively $\h$) is positioned in $G$ (respectively $\g$), and the action of the component group $H/H^0$. 
More precisely, $ \NC_{\g,\h} $ and $\CC_{\g,H}$  are determined by how
$H$ acts on 
the space $\z(\h) \cap [\g, \,\g]$ and on how certain bilinear forms on $\g$, constructed using the Killing forms on the simple factors of the 
semisimple part $[\g,\, \g]$, behave when restricted to the subalgebra $\h \cap [\g, \,\g]$. 

\begin{remark} We will now make remarks on the computational aspects of the third and fourth Betti numbers using the main results; see Sections 4.1, 4.2, and 4.3. for results in this regard.
In the most general setting, where $G$ is a connected compact Lie group and $H$ is a closed (but not necessarily connected) subgroup, the computations of the dimensions of the spaces $ \NC_{\g,\h} $ and $\CC_{\g,H}$ 
essentially reduces to the problem of finding the rank and nullity of a natural linear map associated with the pair $(\g, H)$; see  Sections \ref{associated matrix}. 
In a  specific set-up, when $G$ and $H$ are given, all the quantities involved in the above two results are explicit, computable.
This is the best that can be said in the most general setting, though in a restricted set-up, sharper  results may be obtained, see Section \ref{application-topological-invariant}, Section \ref{computation in special setting}.   \qed
\end{remark}

\subsubsection{Applications and a new topological invariant}\label{application-topological-invariant}

The next result gives rise to an elegant way of computing a topological invariant associated to a large class of compact homogeneous spaces.  

For a semisimple Lie algebra $\k$, {\it let $\# \k$ denote the number of simple factors in $\k$.} See Definition \ref{definition-invariant-min-ideal} for the notation $\#([\h,\,\h],\,H)$.

\begin{theorem}\label{thm-invariant}
 Let $G$ be a compact semisimple Lie group, and $H$ be a closed  subgroup of $G$ which is not necessarily connected.  Then the following hold:	
\begin{enumerate}
\vspace{.12cm}
\item \label{cor1.3-3}  The equality, $$ \dim   H^4(G/H )\,- \,\dim  H^3(G/H )\,=\,   \dim  S^2\big(\z(\h)^* \big)^{H/H^0} \,+\, \#([\h,\h],H)    -\,\# \g$$ holds.
\item \label{cor1.3-4}   \label{H-0-semisimple-cor1.3} If moreover $H^0$ is  semisimple,  the equality
$$\dim   H^4 (G/H  )- \dim   H^3 (G/H  )
\,= \,  \#(\h,H)\,- \, \# \g	$$
holds. Furthermore, if either $ H $ is connected or all the simple factors of $ H ^0$ are pairwise non-isomorphic then we have
$$\dim   H^4 (G/H  )- \dim   H^3 (G/H  )
\,= \,  \#\h\,- \, \# \g\,.	$$

\item \label{cor1.3-5}  If moreover $H$ is a toral subgroup, then  	
$$ \dim   H^4(G/H )\,\,- \,\dim  	H^3(G/H )\,\,=\, \frac{\dim  \h(\dim  \h+1) }{2} \,-\, \# \g\,.$$
\end{enumerate}	
\end{theorem}

In the above results we do not impose any restriction on the rank of the ambient group $G$ and closed subgroup $H$. In a  special restrictive case when rank$(H)=$  rank$(G)$, the conclusions
 of Theorem \ref{thm-invariant}\eqref{cor1.3-4}  and \eqref{cor1.3-5} follow from \cite[Theorem 26.1]{Bo}, \cite[Proposition 26.1]{Bo}, respectively,  using  description of the cohomology algebras of $G$ and $H$.     

Theorem \ref{thm-invariant} leads to the next corollary where we obtain an  interesting fact 
that, for the class of homogeneous spaces obtained by quotienting a semisimple Lie group by a  semisimple subgroup, the difference between the number of simple factors of the ambient group and that of the subgroup, becomes a topological invariant.

\begin{corollary}\label{cor-homotopy-inv}
Let $G$  be a compact semisimple Lie group, and $H$  be a closed semisimple subgroup of  $G$.  	
\begin{enumerate}
	\item \label{corollay-in-the-intro}
Let $ G_1 $ be another  compact semisimple Lie group, and $H_1$  be a closed semisimple subgroup of   $G_1$.   Let $\g,  \g_1, \h$ and $\h_1$ be the Lie algebras of $G, G_1,  H$ and $H_1$, respectively. Assume  that 
	 $$\# \h_1\,- \,\# \g_1 \,\,\neq\,\, \# \h\,- \, \# \g \,.$$ 
Then $G_1/ H_1$ is not homotopically equivalent to $G/H$.
\vspace{.2cm}
\item  \label{cor-hpt-inv}	
Let $T$ be a maximal torus of $G$.  Then $G/T$ can not be homotopically equivalent to $G/H$.
\end{enumerate}
\end{corollary}

In Section \ref{appl-thm-noncompact} we derive consequences in the set-up
of homogeneous spaces which are not-necessarily compact.
In specific contexts, we obtain more precise information on the Betti numbers
which are interesting on their own right. 

\begin{corollary}\label{cor-g-semisimple}
Let $G$ be a compact semisimple Lie group and $H$ be a closed subgroup of $G$  which is not necessarily connected and $\dim H\,>\,0$. 
 Then the following holds:
	\begin{enumerate}
		\item \label{cor1.3-1}
 $\dim  H^3(G/H ) \,\leq\,\#\g-1$.    
  Moreover if $ G $ is simple, then $ H^3(G/H  ) \,=\,0$. 
 
 \vspace{0.1cm}

\item \label{cor1.3-2}   
$\dim  H^4(G/H  )\, \leq \,
\dim  S^2\big(\z(\h)^* \big)^{H/H^0} \,+\,\, \#([\h,\,\h],\,H)$ $ - 1$.   
		 In particular, for  specific $G$ and  $ H $ we have the following bound:  \vspace{0.1cm}
\begin{enumerate}
\item \label{cor1.3-2a}  If $ H^0$ is semisimple, then $\dim  H^4(G/H  )\, \leq \, \,\#(\h,H) \,-1 \,.	$
\vspace{.15cm}
\item \label{cor1.3-2b}  If $ G $ is simple and  $ H^0$ is semisimple, then 	
	  $\dim  H^4(G/H,\R)=   \#(\h,H)  - 1 $. 
		\vspace{.15cm}
 \item \label{cor1.3-2c}  If $ G $ is semisimple and	$H^0$ is  simple, then $\dim  H^4(G/H  )\, =\,  0$.
 \end{enumerate}
 \end{enumerate}	
\end{corollary}

The following result illustrates a subtle  consequence of the fact that  component groups can play a role in the computation of cohomology.  
Put differently, when $H$ is connected, a finite subgroup $F$ of $G$ that normalizes $H$ can influence the lower bounds of $\dim H^3(G/H)$.
 We refer to Definition \ref{definition-invariant-min-ideal} for the notation $\IC([\h, \h])$.

\begin{corollary}\label{cor-finitegp-in-intro}
	Let $G$ be compact semisimple, and $H$ be a closed connected subgroup. Let $m := \#\g$ and $p := \dim S^2(\z(\h)^*)$.  If  $F \subset G$ is a finite group that normalizes $H$ then 
		\[
		\dim H^3(G/H) = \dim \NC_{\g,\h} \geq m - (k + p).
		\]
		where $k$ denote the number of orbits for the associated action of $F$ on $\IC([\h, \h])$.
\end{corollary}

\begin{remark}	
	In this remark we will include a vanishing result which follows from \cite{BC, BCM}, and Corollary \ref{cor-g-semisimple}\eqref{cor1.3-1} 
	and \eqref{cor1.3-2c} : Let $ G $ be a simple compact Lie group, $ H\,\subset\, G$ be a closed subgroup such that $ H^0 $ is also simple. Then 
	$H^i(G/H) \, =\, 0 \, $ for $1\,\leq\, i\,\leq \,4$. In this case the vanishing of $H^3 (G/H^0)$
	also follows from \cite{A} and \cite{LW}.

If $ G $ is a simple non-compact Lie group, and $ H\,\subset\, G$ is a closed subgroup such that $ H^0 $ is  simple, then from \cite{BC, BCM} and from our results above
 we have the following lower bound: $\dim H^i(G/H)\leq 1$ for $i=1,2$,   $\dim H^3(G/H)\leq 2$,  $\dim H^4(G/H)\leq 3$.  
 \qed
\end{remark}

The next result is another interesting consequence of Corollary \ref{cor-g-semisimple}\eqref{H-0-semisimple-cor1.3}.

\begin{corollary}\label{dim-7} Let $G$ be a connected compact Lie group, and let $H$ be a closed subgroup which is not necessarily connected. Assume that $G$  and $H^0$ are both semisimple with $\dim  G/H \,=\, 7$.  Further if either $ H $ is connected or all the simple factors of $ H $ are pairwise non-isomorphic, then $\# \g\,=\, \# \h$.
\end{corollary}

It should be noted that Corollary \ref{dim-7} also  follows from the classification of seven dimensional compact homogeneous spaces as in \cite[Theorem 2]{Go}.

\subsection{Our strategy in proving the main results} \label{strategy}
Our strategy in proving the main results involves the well-known Cartan’s theorem concerning the cohomology of homogeneous spaces.
In fact, we employ an equivariant version of the Cartan's theorem Theorem \ref{Cartan-thm-disconnected} (see also \cite[Theorem 3.5]{BCM2}) which is suitable in
our framework. This version was expected to be known, and a proof is given in
\cite{BCM2} as we could not find one in the existing literature.  
Although bringing out simple computable  expressions of the cohomologies in higher dimensions from the Cartan's theorem seem prohibitively intricate we are 
able to do it in the third and fourth dimensions using techniques that involve breaking up suitably certain spaces related to 
the lower dimensional parts of the Koszul complex, as well as an intricate analysis of the action of the component group $H/H^0$ on certain 
relevant subalgebras associated to the pair $(\g, \,\h)$.
An explicit choice of transgression plays a vital role in determining certain differential maps in the Koszul complex; see Section \ref{sec-description-third-coh}.

\subsection{Comparison with prior results}\label{comparison}

Here we make some remarks comparing the mathematical and chronological aspects
of our results with those of \cite{A} and \cite{LW}
as these works also deal with the third Betti number of a compact homogeneous space.  
\begin{itemize}
\item The second part of the Corollary \ref{cor-g-semisimple}\eqref{cor1.3-1} was previously proved  in \cite{A} in 1990 where the author considered restricts to the case when $G$ is simple. In the paper \cite{LW} a description of the third Betti numbers are given when $G$ is compact semisimple and $H$ is connected closed subgroup.
In this paper we deal with the most general case, that is the case when $G$ is reductive and $H$ is not necessarily connected, and moreover we find a clear description of both
the third and fourth Betti numbers simultaneously.

\item It should be noted that in \cite{LW}
 the well known Hodge theorem and in particular the harmonic forms in the set-up of compact oriented homogeneous spaces
 are employed. In our approach we use a modified version of Cartan's theorem which
may be of independent interest for possible future use as it is formulated in the framework of most general compact homogeneous spaces.

\item In \cite{LW} the authors deal with only the third Betti number where the ambient group $G$ is semisimple and $H$ is connected.
Compared to \cite[Proposition 4.1]{LW}, our proof of Theorem \ref{thm-3rd-coh-general} on the third Betti number is more intricate due to several key factors. We provide a unified treatment of both the third and fourth cases within the most general framework of homogeneous spaces, where $G$ may be reductive but not necessarily semisimple, and $H$ may be disconnected.

In this broader setting, the non-semisimplicity of $G$ and the potential disconnectedness of $H$ introduce significant complexities. 
We stress that the argument simplifies considerably under the stricter assumptions that $G$ is semisimple and $H$ is connected, as many structural decompositions and technical subtleties become unnecessary in that restricted setting.

Furthermore, bringing in the potential disconnectedness of $H$ in the general setting plays an important role because
even when $H$ is connected, if it arises as the identity component of a larger disconnected group $H'$, the action of $H'/H$ on certain spaces influences $H^3(G/H)$; see Corollary \ref{cor-finitegp-in-intro}.
\end{itemize}
 
 For some more specific comparison (and a suggested modification) of certain results from \cite{LW} with those obtained here is provided in  Remark \ref{remark-LW}.

For the purpose of comparing the chronology, it should be mentioned that Theorem \ref{thm-3rd-coh-general-non-cpt}, Theorem \ref{thm-4th-coh-general-non-cpt}, and Corollary \ref{cor-g-semisimple}\eqref{cor1.3-1} were first announced at the {\it 87th Annual Conference of the Indian Mathematical Society} on December 5, 2021, and subsequently at the online conference \href{https://sites.google.com/view/world-of-groupcraft-2}{\it World of Group Craft -II}  on September 2, 2022. The corresponding talks are available at the 30-minute mark of the  YouTube video  \url{https://www.youtube.com/watch?v=XsRXpE8npuo}  and at  \url{ https://www.youtube.com/watch?v=UGNBlfX9FXY}.

\subsection{Organization of the paper} 
This paper is organized as follows. In Section \ref{background-Cartan-thm} we recall some necessary background. 
Section \ref{section-3rd-4th-cohomology} is devoted to deriving 
our main results, namely Theorems \ref{thm-3rd-coh-general} and \ref{thm-4th-coh-general}. 
Section \ref {appl-thm} is devoted to computation and application of the main results. 
After considering some general computation for compact Lie groups, we present certain  reduction procedures  that  facilitate the cohomology computation;   
see Sections \ref{computation-compact-case}, \ref{first-reduction}, \ref{associated matrix}.
Also numerous applications of our main results are obtained in  Section \ref {computation in special setting}.
 Section  \ref{appl-thm-noncompact} deals with the homogeneous spaces which are not necessarily compact. 

In Appendix  \ref{appendix}, we clarify  certain issues that occur with regard to the relation between maximal compact  subgroups of a semisimple Lie group and those of the simple factors. In Appendix \ref{appendix} we also  list the simple Lie groups according to their types (a relevant notion which is introduced earlier).

\section{Notation and background}\label{background-Cartan-thm}
The purpose of this section is to recall certain technical but unavoidable background required in proving our main results.

\subsection{Notation}\label{notation}
Here we fix some notation that will be used throughout this paper. Subsequently, a few specialized notations are introduced as and when they occur.

Let $V$ be a finite dimensional vector space over $\R$.
The {\it space of symmetric (respectively,  alternating)} $\R$-valued $ \R$-linear  maps in $k$-variables  is denoted by $S^k(V^*)$ (respectively, $\wedge^k V^*$).  
As usual,   we denote the symmetric (respectively,  alternating)
 product of such symmetric (respectively,  alternating)  $\R$-linear maps by `$\vee$' (respectively, by `$\wedge$').   
 It should be noted that the scaling factors in this products vary between authors.  We will follow \cite[p.\,5]{GHV-2} in this regard.

The following convention will be used throughout this paper : If $V$ is a vector space
with a direct sum decomposition $V = V_1 \oplus \cdots \oplus V_s$ where $V_i \subset V$ are subspaces then for $ 1 \leq t \leq s$ the space $S^2 (V^*_t)$ will be identified with
the subspace $\{ B \in S^2 ( V^*) \mid B ( V_i, V_j) = 0, \text{ when }
i \neq t, \text { or } t \neq j \}$ of $S^2 (V^*)$.

Sometimes, for 
notational convenience, the Lie algebra of a Lie group $G$ is also denoted by 
${\rm Lie} (G)$.  The connected component 
of $G$ containing the identity element is denoted by $G^{0}$. The {\it center} of a group $G$ is denoted by 
$Z(G)$, while the {\it center} of a Lie algebra $\g$ is denoted by $ \z (\g)$.

If $\k$ is a Lie algebra, $W$ is a $\R$-vector space and  $\psi \,:\, \k \,\longrightarrow\, {\rm End}  W $ is a representation of $ \k$, then  define
\begin{equation}\label{wk}
	W^\k\, :=\,  \{ \, w \,\in\, W \  \mid\ \psi (x) w \,=\, 0\ \text{ for all }\ x \,\in\, \k \}\, .
\end{equation}
For a graded vector space   $E:= \bigoplus_{i\in \N} E^i$, where the vector space $E^i$ consists of the homogeneous elements of  degree $i$, $E^+\,:=\, \bigoplus_{i>0} E^i$.
For a semisimple Lie algebra $\k$ over $\R$ (respectively, $\C$), let $ \#\k $ denotes the {\it number of simple factors of } $\k$ over $\R$ (respectively, $\C$).

It should be noted that although various notions from  the book \cite{GHV-3} are used in Section
\ref{section-3rd-4th-cohomology},  the corresponding notations in \cite{GHV-3} differ from those 
in this paper; for instance, $W^\k$ in \eqref{wk} corresponds to $ W_{\psi =0} $ in \cite{GHV-3}.

We next set another notation which will be used in several places later.
Let $\g$ be a reductive Lie algebra and $\b$ be an ideal in $[\g, \, \g]$. Let $\b'$ be the unique ideal in $[\g, \, \g]$ such that $[\g, \, \g] = \b \oplus \b'$. We will denote the {\it projection 
to the factor $\b$ from $\g$  with respect to the decomposition $\g = \z (\g) \oplus \b \oplus \b'$ by }
\begin{equation}\label{notation-projection}
\pi^\g_\b : \g \to \b
\end{equation}
When there is no possibility of any confusion we will denote the above projection simply by $\pi_\b$.

Let ${\Bi}_{\b}$ denote the Killing form on $\b$.  We define $\widetilde {\Bi}^\g_{\b}$ on $\g$ by
\begin{equation}\label{definition-B}
   \widetilde {\Bi}^\g_{\b} (X,\,Y)\,:=\, {\Bi}_{\b} (\pi^\g_\b (X), \pi^\g_\b (Y)) \,, 
\end{equation}
for all $X, Y  \in \g$. Let now $G$ be a Lie group, which is not necessarily connected
such the associated Lie algebra is $\g$. Then the adjoint action of $G$ on $\g$ induces an action on the space $S^2(\g^*)$. 
It is then clear that $\widetilde {\Bi}^\g_{\b}\,\in\, S^2(\g^*)^{G^0}$.

\subsection{{Background}}
Here we recall some background required for the next sections. 

\subsubsection{Primitive elements}\label{secps}
Let $\g$ be a reductive Lie algebra. Consider the homomorphism $\mu_\wedge\,\colon\, \wedge \g \ox \wedge\g \,\longrightarrow\, \wedge\g$ defined by $X\ox Y\,\longmapsto\, X\wedge Y$. The ${\rm ad}$-representation of $ \g $ naturally provides another representation $ \th\,:\, \g \,\lto\, {\rm End}\,\wedge \g$, and with respect to this representation we may form the subspace
$(\wedge \g)^\g $ of $ \wedge \g$.  It is clear that $\mu_\wedge ( 
(\wedge \g)^\g \ox (\wedge\g)^\g )\, \subset (\wedge\g)^\g$. Let 
$$
(\mu_\wedge)^\g \,\,\colon\, (\wedge \g)^\g \ox (\wedge\g)^\g \,\lto\, (\wedge\g)^\g
$$ 
denote the restriction of the map $\mu_\wedge$ 
to $(\wedge \g)^\g \ox (\wedge\g)^\g\, \subset\, \wedge \g \ox \wedge\g$.
Let $\gamma_\g\,\colon\,(\wedge\g^*)^\g \lto (\wedge \g^*)^\g \ox (\wedge\g^*)^\g$ 
be the dual of the homomorphism $(\mu_\wedge)^\g$.
An element $\phi\,\in\, (\wedge^+ \g^*)^\g$ is called {\it primitive} if 
$$\gamma_{\g} (\phi)\,=\, \phi \ox 1 + 1\ox \phi.$$

The {\it subspace of primitive elements} will be denoted by 
$P_\g$. Then $P_\g$ forms a graded subspace 
$P_\g \,=\, \sum_j P^j _\g$ of $(\wedge^+\g^*)^\g$. 

\subsubsection{Transgression}\label{sec-transgression}

\begin{definition}\label{definition-weil-alg-inv}
The {\it Weil algebra} $W(\g)$ of a reductive Lie algebra $\g$ is defined by  
$$
W(\g):= W(\g,\,\g)\,=\,  S(\g^*) \ox \wedge \g^*  \,.
$$
We set $W(\g)^\g := W(\g,\,\g)^\g$, the invariant subalgebra of  $W(\g,\,\g)$.
\qed
\end{definition}

\begin{definition}\label{defn-transgression}
A {\it transgression } in the differential algebra $W(\g)^\g $ is a linear map   
$$\tau\,\colon\, P_\g\,\lto\, S(\g^*)^\g$$
satisfying the following two conditions (see \cite[\S~6.13, p. 241]{GHV-3}):
\begin{enumerate}
\item $\tau$ is a graded map of degree $1$.

\item For every $\phi \,\in\, P_\g$, there is an element $\omega \,\in\, (W(\g)^+)^\g $ such that $$
\dl_W(\omega)\,=\, \tau \phi\ox 1 \qquad \text{ and } \qquad 1\ox \phi - \omega \,\in\, \sum_{j\geq 1}(S^j(\g^*)\ox \wedge \g^*)^\g.$$
\end{enumerate}
We denote by ${\rm Trans}(\g)$ the {\it space of all transgressions} for the Lie algebra $\g$. 
\qed
\end{definition}

\subsubsection{Koszul Complex}\label{sec-Koszul-complex}

Let $(\g,\,\h)$ be reductive Lie algebra pair with the inclusion map $j\,:\, \h\,\longrightarrow\, \g$.  Let 
 \begin{align}\label{defn-j-star-map}
j^* : \g^* 
\lto \h^* \quad \text{ and } \quad j^* : S(\g^*) 
\lto S(\h^*)
\end{align}
denote the dual map from $\g^*$ to $\h^*$, and the corresponding natural extension
from the symmetric algebra $S(\g^*)$ to $S(\h^*)$, respectively (we 
continue to denote both the maps by the same notation as long as there is no
confusion). Then $j^* ( S(\g^*)^\g) 
\subset S(\h^*)^\h$.
Let $\tau \,\colon\, P_\g \,\lto\, S(\g^*)^\g$ be a transgression.
Define
\begin{align}\label{defn-sigma-map}
\sigma\,  :=\,j^* \circ\tau\, \, \colon\, P_\g \,\lto\, S(\h^*)^\h\, .
\end{align}
Consider the {\it Koszul complex} $(S(\h^*)^\h \ox (\wedge P_\g), \,\nabla_\sigma)$ associated to the transgression
$\tau$, where the differential $\nabla_\sigma$ is given by (see \cite[\S~10.8, p.~420]{GHV-3})
\begin{align}\label{diff-nabla}
\nabla_\sigma (\psi\ox ( \phi_0\wedge \cds \wedge \phi_p)) \,: =\, \sum_{j=0}^p (-1)^j\big( \psi\vee \sigma(\phi_j) \big)
\ox\big( \phi_0 \wedge \cds\wedge \widehat{\phi_j}\wedge \cds\wedge \phi_p\big)\,,     
\end{align}
 for all  $\psi\, \in \, S(\h^*)^\h,\,  \phi_i\,\in\, P_\g.$  It is important to keep in mind  that the Koszul complex is a graded algebra 
 with respect to the total degree, $\deg \psi\,=\, 2k$ for all non-zero $\psi\,\in\,S^k( \h^*)^\h$, and the  elements of $ P_\g $ have the usual degrees.

It can be shown that $ \nabla_\sigma ^2\,=\,0$. Thus the Koszul complex $(S(\h^*)^\h\ox (\wedge P_\g),\,\nabla_\sigma)$ 
is indeed a complex. The differential $\nabla_\sigma$ depends on the choice of the transgression $\tau$. However, the cohomology of the complex is independent of the choice of $\tau$.

\subsubsection{Cartan's theorem for disconnected $H$}
Let $G$ be a compact connected Lie group and $H\,\subset\, G$ a closed, but not necessarily a connected subgroup. One of the main ingredient of this paper is the 
following version of the celebrated  Cartan's Theorem( \cite[Theorem 5]{Ca2},  \cite[Theorem 25.2]{Bo}), when $H$ is not necessarily connected. Although this version 
is expected to be known, but as no proofs could be found in the existing literature, for the sake of completeness, we give one in \cite{BCM2}.

\begin{theorem}[{\cite[Theorem 3.5]{BCM2}}]\label{Cartan-thm-disconnected}
	Let $G$ be a compact connected Lie group and $H\,\subset\, G$ a closed, but not necessarily a connected subgroup. The following vector spaces are isomorphic:
	$$
	H^*\big(G/H  \big)\,\simeq \,
	H^*\big( (\wedge \g^*)^H_{i_\g|_\h=0},\, \dl_\g\big)\, \simeq\,\,    H^*\big( ({S(\h^*)}^H\ox \wedge P_{\g}),\,  \nabla_{\sigma} \big)\, . 
	$$
\end{theorem}

\subsubsection{Invariant decompositions associated to  
the Lie algebra of a compact Lie group and a subalgebra}\label{prep section} 

Here  we list certain decompositions associated to a Lie algebra and a subalgebra which will be used later.

Let $G$ be a compact connected Lie group and $H$ a closed subgroup. As before, let $ \g:= {\rm Lie}(G) $ and $ \h:= {\rm Lie}(H) $.  
Since $\g$ and $\h$ are reductive, we have
$$
\g\,=\, \z(\g) \,\oplus \, [\g,\,\g], \quad \h\,=\, \z(\h) \,\oplus \, [\h,\,\h] \quad \text{ and }\quad
[\g,\,\g]\cap \h \,=\, [\h,\,\h]\oplus (\z(\h)\cap [\g,\,\g])\, .
$$

Let $\langle \cdot,\,\cdot\rangle$  be a $G$-invariant inner product 
on the $\R$-vector space $\g$.  For subspace $V$ we denote the orthogonal complement by $V^\perp$.

We decompose $[\g,\,\g]$, $ \h $ and $[\g,\, \g] + \h$ as follows :
\begin{align}
 \h \,= \,& \big( ([\g,\, \g] \cap \h)^\perp \cap \h \big)\, \bigoplus\, ([\g,\, \g] \cap \h)      \nonumber \\
= \,& { \big( ([\g,\, \g] \cap \h)^\perp \cap \h \big)}\, \bigoplus \,[\h,\,\h] \,\bigoplus\,
{(\z(\h)\cap [\g,\, \g] )}\,, \label{eq4h-gg} 
\\
 [\g, \,\g] + \h \,
  = \, &\, [\g, \,\g]\, \bigoplus\, \big( ([\g,\, \g] \cap \h)^\perp \cap \h \big)\,    \label{eq4gg.1h} \,.
 \end{align}
  	
For  notational convenience, we set
\begin{align} \label{defn of a}
\a \,&:=\, [\g,\,\g] \cap {\z}(\h) \,,    \\
 \b \,& :=\, ([\g, \,\g] \cap \h)^\perp \cap \h   \,.\nonumber
\end{align} 

It may be observed that 
$\, \b\, = \, \big( (\, [\g ,\, \g] \cap {\z} (\h) ) + [\h,\, \h] \big)^\perp \cap \h
\,=  \, \big( [\g, \, \g] \cap {\z} (\h)\big)^\perp \cap {\z} (\h)$.   \\
  	Then  we have
\begin{align}
 {\z} (\h)& =\, \a\, \oplus \, \b,\label{z(h) decomposition}\\
 \h\,&= \,\a \, \oplus   [\h,\,\h] \,\oplus   \,\b.\label{h decomposition}
\end{align}  	
 It should be noted that  all the subspaces of $\g$ appearing above  are $H$-invariant.

\subsubsection{Some preparatory results} 
Here we record a couple of useful  general observations. 

\begin{lemma}\label{lem-orthog-ideals} 
Let $ \g $ be a Lie algebra over $ \R $. Assume that $\g\,=\,\g_1\oplus\g_2 $ where 	$ \g_1,\,\g_2 $ are ideals in $\g $ and $ \g_1\,=\,[\g_1,\,\g_1] $. Let $\theta \,:\, \g \times \g \,
	\longrightarrow\, \R$ be a $ \g $-invariant bilinear form on $ \g $. Then 
	$ \g_1$ and $ \g_2 $ are mutually orthogonal with respect to $\theta$.
\end{lemma}

As the Lie algebra of a compact Lie group is reductive, we have an immediate corollary.

\begin{corollary}\label{lem-ortho-center-gg}
Let $ G $ be a compact Lie group with Lie algebra $\g$. Consider the adjoint action of $G$ on $\g$, and let  $ \theta $ be a 	$ G $-invariant bilinear form on $ \g $. Then $ \z(\g) $ and $[\g,\,\g]$ are mutually orthogonal with respect to $\theta$. 		Furthermore, if $[\g,\, \g] \,=\, \g_1 \oplus \cdots \oplus \g_r$ with $\g_i$ a simple ideals in $\g$ for every  $i \,= \,1,\, \dots,\, r$, 
	then $\g \,=\, \z(\g) \oplus \g_1 \oplus \cdots \oplus \g_r$ is in fact a decomposition into subspaces which are mutually 
	orthogonal with respect to $\theta$. In particular, if $\theta$ is an inner product, then $[\g, \,\g]^\perp \,=\, {\z} (\g)$. 
\end{corollary}

We next observe a fact which is interesting in its own right as the  subgroup involved is not assumed to be connected.

\begin{lemma}\label{surpring-lem}
Let $ G $ be a connected compact Lie group, and $ H\subset G $ be a closed subgroup which is not necessarily connected. 
Let $ \p\subset \g $ be a subspace such that  ${\rm Ad} (H) (\p) \subseteq \p$. Assume that $\<>$ is a $ G $-invariant inner 
product on $ \g $. Then the Lie group $ H $ acts trivially on the space $ (\p\cap [\g,\,\g])^\perp \cap \p $.
\end{lemma}
\begin{proof}
	Recall that $ \g\,=\,\z(\g) \oplus [\g,\,\g] $. For any $ Z\,\in\, (\p\cap [\g,\,\g])^\perp \cap \p  $, we can write
	$Z\,=\,X+Y $ where $ X\,\in\, \z(\g) $ and $ Y\,\in\, [\g,\,\g] $. Let $ h \,\in\, H $. Then $ {\rm Ad} (h) Z - Z
	\,\in\,  (\p\cap [\g,\,\g])^\perp \cap \p $. On the other hand ${\rm Ad}(h)Y  - Y\, \in \, [\g,\,\g]$. Since 
	$ {\rm Ad} (h) Z - Z  \,=\,  {\rm Ad}(h)Y  - Y $, we  also have $ {\rm Ad} (h) Z - Z  \,\in\, \p\cap [\g,\,\g]$. 
	Thus $ {\rm Ad} (h) Z - Z\,=\,0$. This completes the proof.
\end{proof}

The following is a useful fact on the representations of compact Lie groups, which we state below, skipping the proof.  It  will be needed later. 

\begin{proposition}\label{representation-proposition} Let $K$ be a compact Lie group which is not necessarily connected and $V$ be a 
finite-dimensional vector space over $\R$. Assume that $\phi \,:\, K \,\lto\, {\rm GL} \, (V)$ is  a smooth representation of $K$.  Let $V^K $ 
be the space of fixed points in $V$ under the action of $K$, that is, $V^K\,=\, \{ v \,\mid\, \phi (g) (v)\, =\,v \text{ for all } g \,\in\, K\}$ and let 
$(K -1) V$ be the $\R$-span of the elements of the set $\{ \phi (g) v -v \,\mid\, g \,\in\, K, \, v \,\in\, V\}$. 
	\begin{enumerate}
		\item Then  $(K-1) V$ is the  unique $K$-invariant complement of $V^K$. In particular, 
		$V \,=\, V^K  \oplus (K-1) V$.
		\item  Let $ \Bi $ be any $ K $-invariant  bilinear form on $ V $. Then $V \,=\, V^K  \oplus (K-1) V$ is an  orthogonal direct sum with respect to $ \Bi $.
	\end{enumerate}		
\end{proposition}

\section{Proofs of the main results}\label{section-3rd-4th-cohomology}
In this section, we will prove Theorem \ref{thm-3rd-coh-general} and Theorem \ref{thm-4th-coh-general}.  We will use an equivariant version of Cartan's theorem in the proofs, namely Theorem \ref{Cartan-thm-disconnected} (a proof of this version
is given in \cite{BCM2})

It is important to keep in mind that the following convention and notation will be used in rest of this section.

\begin{enumerate}[label = {{\bf  C}.\arabic*}]
\item~ \label{convension1}	
In what follows whenever we have a vector space $V$ and subspaces $V_1,\, \dots,\, V_k$ of it such that
$V \,=\, V_1 \oplus \cdots \oplus V_k$, for every $i \,=\, 1,\, \dots,\, k$
the dual $V_i^*$ is identified with the subspace of $V^*$ consisting all elements that vanish on
$V_j$ for every $j\, \not=\, i$. With this identification, we have the 
direct sum decomposition $V^* \,=\, V_1^* \oplus \cdots \oplus V_k^*$. Further,
for all $1 \,\leq\, i,\, j \,\leq\, k$, let
$V_i^* \vee V_j^*\, \subset\, S^2 (V^*)$ be the subspace defined by
$V_i^* \vee V_j^* \,:=\,{ \rm Span}\{f \vee g \, \mid\, f \,\in\, V_i^*,\, g \,\in\, V_j^*\}$. Then
we have $S^2 (V^*) \,= \,\bigoplus_{i \leq j} V_i^* \vee V_j^*$. 
We will use this as well as the analogous decompositions of
$S^r (V^*)$ and $\wedge^r V^*$. 

\medskip
\item~\label{convension2}
 In this section we will consider a general connected compact Lie group $G$ and a closed subgroup $H$, which is not necessarily connected. We fix a $G$-invariant inner product $\<>$ on $\g$. 
Recall that $P^1_\g\,=\, \{f\in \g^* \,\mid\, f|_{[\g,\,\g]}\, =\,0\} $. It follows that the coadjoint action of $G$ on $\g^*$ keeps $P_\g^1$ invariant.
We now choose a $G$-invariant inner product $\<>^*$ on $P_\g^1$ (the latter inner product
may be obtained as the restriction of a $G$-invariant inner product on $\g^*$).

\medskip
\item ~ \label{convension3}
Let $j^* : \g^* \,\lto\,  
\h^*$ be  the map  as in \eqref{defn-j-star-map}.
Observe that $j^* ( P_\g^1) \subset (\h^*)^H$. Let 
$\Re \,:\, P_\g^1\,\longrightarrow\, (\h^*)^H$ denote the restriction 
of $j^*$ to $P_\g^1$. Thus $\Re$
maps $f\,\longmapsto\, f|_\h$ for all
$f\,\in\, P_\g^1$. 
\end{enumerate}

As the coadjoint action of $H$ on $\g^*$ keeps $P_\g^1$ invariant, it is easy to see that $\ker \Re$ is a $H$-invariant subspace of $ P_\g^1$. Moreover, we have the following $H $-invariant orthogonal decomposition:
\begin{align}\label{decomposition-Pg1-ker-Rgh-ker-perp}
P_\g^1 \,=\,   \ker \Re  \oplus ( \ker \Re )^\perp.
\end{align}

Consider  the map $\nb^1$:
\begin{align}\label{map-nbl1}
\nb^1\colon 1\ox P^1_\g\, \lto\,  (\h^*)^H\ox 1 \quad \text{ given by } \quad \nabla^1( 1\ox f ):=  f|_\h\ox 1\,.
\end{align}
\begin{align}\label{ker-nbl}
\text{Observe that } \, \,{\rm ker\,} \nb^1  \, \simeq   \ker \Re =    \, \{f\,\in\, \g^* \,\mid \,f|_{[\g,\,\g]+\h} \,=\,0\}\,
\simeq\,\big(\frac{\g}{[\g,\,\g]+\h}\big)^* \,.
\end{align}

\begin{lemma}\label{lem-ker-im-nb1}
Let $\nb^1\,\colon\, 1\ox P^1_\g\, \lto\,  (\h^*)^H\ox 1$ be as in \eqref{map-nbl1}. Then 
$$
\dim    {\rm Im} \nb^1 \,=\, \dim   P_\g^1 - \dim   {\ker} \nb^1    \,=\, \dim  ([\g,\,\g]+\h) - \dim   [\g,\,\g]  \,  =\, \dim   \b.
$$
In particular, $~{\Im {\nb}^1}  = (\b^*)^H\ox 1$ and  $(\b^*)^H = \b^*$.  
\end{lemma}

\begin{proof}
See \cite[Lemma 4.1]{BCM2} for a proof. 
\end{proof}

We will now return to the direct-sum decompositions as in Section \ref{prep section} and observe some more facts associated with  
actions of $H$ on some of the summands. These facts will be used later.

Recall that all the direct summands in Section \ref{prep section} are $H$-invariant under the adjoint action. 
Using either Lemma \ref{lem-ker-im-nb1} or Lemma \ref{surpring-lem}, it follows that, although $H$ is not necessarily connected, it fixes all the elements of $\b^*$.  
It is now easy to observe
that $\h^H \,\subset\, {\z} (\h)$. Thus $\h^H\,=\, {\z} (\h)^H$ and $[\h, \, \h]^H = 0, (H-1)[\h,\, \h] \,=\, [\h,\, \h]$. As $\b \,=\, \b^H
\oplus (H-1) \b$ and $\b\,=\, \b^H$, it follows that $(H-1) \b \,=\, 0$. Thus using \eqref{h decomposition} we obtain 
\begin{align}
	\h \,=\, \a^H \oplus (H-1) \a \oplus \b \oplus [\h, \,\h]. \label{h decomposition refined}
\end{align}
Moreover, we have
\begin{align}
	(H-1) \h \,=&\, (H-1) \a \oplus [\h, \,\h], \label{h decomposition component1}\\
	(H-1) \z(\h) \,=\, (H-1) \a  \, \,&\text{ and }\, \h^H \,=\, {\z} (\h)^H\, =\, \a^H \oplus \b.\label{h decomposition component2}
\end{align}

It may be further observed that the decompositions as in \eqref{h decomposition refined}, \eqref{h decomposition component1} and \eqref{h decomposition component2} are $H$-invariant and orthogonal with respect to the  ($G$-invariant) inner product $\<>$  introduced in Section \ref{prep section}.

\subsection{Description of  third cohomology}\label{sec-description-third-coh}
Here we will prove Theorem \ref{thm-3rd-coh-general}.
We follow the strategy as outlined in Section \ref{strategy}.

The differential  $\nabla^2 \,\colon\,(\h^*)^H\ox 1 \ \bigoplus\  1\ox\wedge^2 P_\g^1\ \lto\  (\h^*)^H\ox P_\g^1 $ is given by 
\begin{equation}\label{def-nabla2}
	\begin{split}
		\nabla^2(\psi \ox1)\ &=\ 0 \,\, , \hspace{4cm} \,  \psi\,\in\, (\h^*)^H  \,,\\
		\nabla^2( 1\ox f_1\wedge f_2) \,&=\,  f_1|_\h\ox f_2 \, -\,  f_2|_\h\ox f_1\,\, , \quad \, \,f_1,\,f_2\,\in\, P^1_\g .
	\end{split}
\end{equation} 

The map $\nabla^i$ for $i=3,4$ is defined in \eqref{def-nabla3}, \eqref{def-nabla4}.
We first compute  $ {\ker }\nb^2 $ and $ {\Im}\nb^2 $.  Recall that $\Re$ be as in \eqref{convension3}.     Let 
\begin{align}
\Ab_1\,&:=\,\text{Span}\{ f_1|_\h\ox f_2 \, -\, f_2|_\h\ox f_1 \,\,\mid\,\, f_1,f_2\,\in\, (\ker\, {\Re})^\perp \}\,,  \label{alt-part-ox}\\ 
	\Sb_1\,&:=\,\text{Span}\{ f_1|_\h\ox f_2  \,+\, f_2|_\h\ox f_1 \,\,\mid\,\, f_1,f_2\,\in\, (\ker\, {\Re})^\perp \}\,.\label{sym-part-ox}
\end{align}

\begin{lemma}  \label{lem-ker-im-nb2}
Let $\Ab_1, \Sb_1$ be defined as above. Then the following statements hold:
\begin{enumerate}
\item $ {\ker} \nabla^i|_{1\ox (\wedge^i  P_\g^1)}\,\simeq \, 1\ox ( \wedge^i{\ker}{\Re})\,\simeq \, \wedge ^i\Big(  \frac{\g}{[\g,\,\g]+\h}\Big)^*\,,\quad \text{ for } i = 2, 3,4.$	\vspace{1.2mm}
	
\item $ \dim   \Im\nb^2 \,=\, \dim  (\wedge^2\, P_\g^1)  \,\,-\,\, \dim  (\wedge^2\, \ker {\Re}) \vspace{1.5mm} \\   \,=\,     (\dim  \, P_\g^1 - \dim   \, \ker {\Re})    (\dim  \, P_\g^1 +\dim   \, \ker {\Re} -1)  /2  $. \vspace{1.2mm}

\item  	${\Im \nb^2} \,=\,\big(  \b^* \ox \ker\, {\Re}\big) \bigoplus  \Ab_1$.\vspace{1.2mm}

\item   $ \b^*\ox (\ker\, {\Re})^\perp \,=\, \Ab_1\oplus \Sb_1\,.$\vspace{1.0mm}
\item  ${\Im \nb^2} \bigoplus (\a^*)^H\ox  P_\g^1 \bigoplus \Sb_1 =   (\h^*)^H \ox P_\g^1\,$.
\end{enumerate}
\end{lemma}

\begin{proof} 
In view of Lemma  \ref{lem-ker-im-nb1}, the map $ \nabla^i|_{1\ox (\wedge^i  P_\g^1)}\,\colon   1\ox (\wedge^i  P_\g^1) \, \lto \,\b^* \ox (\wedge^{i-1}  P_\g^1)$ is given by 
	
	$\nabla^i(1\ox f_1\wedge \cds \wedge f_i )\,:= \sum_{j=1}^i (-1)^{j-1} \Re{(f_j)}\ox( f_1\wedge \cds  \wedge 	\widehat{f}_j  \wedge \cds  \wedge f_i)$, for $ i=2,3,4\,.$
	
	\noindent
	Then 	$\ker  \nabla^i|_{1\ox (\wedge^i  P_\g^1)} \, \simeq \,1\ox (\wedge ^i \ker \Re) $. Now   $(1)$ follows from \eqref{ker-nbl}.  $(2)$ follows from $(1)$.    
	Clearly,  $\big(  \b^* \ox \ker\, {\Re}\big) \bigoplus  \Ab_1\, \subseteq {\Im \nb^2}$. Since $\dim  \big(  \b^* \ox \ker\, {\Re}\big) \bigoplus  \Ab_1 \, =\,  \dim   \Im\nb^2 $, $(3)$ follows.   
	To prove $ (4) $, we will use the linear isomorphism $ {\Re}\colon  (\ker\, {\Re})^\perp \, \lto\, \b^*$.
	Finally, in view of the convention mentioned at \eqref{convension1} it follows from the decomposition \eqref{h decomposition component2} that $(\h^*)^H\,=\, (\a^*)^H\oplus \b^*$. Now $(5)$  follows from  \eqref{decomposition-Pg1-ker-Rgh-ker-perp}, (3) and (4).
\end{proof}

\begin{lemma}\label{lem-equivalence-G-inv-forms}
	Let $G$ be a connected compact Lie group, and $ H\subseteq G $ be a closed subgroup which is not necessarily connected.    
	Let $\g$ be the Lie algebra of $G$ and $\h \subset \g$ be the Lie subalgebra corresponding to $H$.	 Let $\Bi$ be a $G$-invariant symmetric
	bilinear form on $\g$. Let $ \a \,:=\,\z(\h)\cap [\g,\,\g] $ as in \eqref{defn of a}. Then the
following are equivalent.
	\begin{enumerate}
		\item 	${\Bi}|_{ ( \h\,\cap \,[\g,\,\g] ) \,\times \,(\h\,\cap [\g,\,\g])} = 0$.
		\item ${\Bi}|_{ \big( (\z(\h)\cap [\g,\,\g]) + [\h,\, \h] \big) \times \big( (\z(\h)\cap [\g,\,\g]) + [\h, \,\h] \big)} = 0$.
		\item ${\Bi}|_{\a\times\a} =0 \text{ and } {\Bi}|_{ [\h, \,\h] \times [\h,\, \h] } = 0.$ 
		\item ${\Bi}|_{\a^H  \times \a^H} =0, \,\, {\Bi}|_{(H-1)\a \times (H-1)\a} =0 \text{ and } {\Bi}|_{ [\h,\, \h] \times [\h,\, \h] } = 0.$
		\item ${\Bi}|_{\a^H  \times \a^H} =0 \text{ and }  {\Bi}|_{(H-1)\h\times (H-1)\h} =0$.
	\end{enumerate}
\end{lemma}	 
\begin{proof} Observe that  $(H-1) \h\, =\, (H-1) \z (\h) \oplus [ \h,\, \h]$ is a Lie subalgebra of $\h$. As ${\Bi}|_{\h \times \h}$ is $H$-invariant
	and $\a^H \subset \z (\h)^H$ it follows using Lemma \ref{lem-orthog-ideals} and Proposition \ref{representation-proposition} that
	$\a^H, (H-1) \a$ and $[\h, \,\h]$ are mutually orthogonal with respect to $\Bi$. This completes the proof.
\end{proof}

\begin{remark}
	One may use Lemma \ref{lem-equivalence-G-inv-forms} to have alternative descriptions of $ \NC_{\g,\h}$ involving $\z(\h)\cap [\g,\,\g]$
	and $[\h,\,\h]$. Note that the definition of  $\NC_{\g, \h}$ does not involve $H/H^0$ and involves only data related to the Lie algebras $ \g$ and $\h $.  On the other hand, the definition of $ \CC_{\g,H} $ does involve the Lie group $ H $ (which is not necessarily connected) along with the Lie algebras $\g,\h$. 
	However, if $ H^0 $ is a semisimple Lie group and all the simple factors of $ H $ are pairwise non-isomorphic, then the component group $ H/H^0 $ does not play any role in the definition of the vector space $ \CC_{\g,H}$; see Lemma \ref{lem-decomposition-of-S(h)-new}.   \qed
\end{remark}

Recall that for a  semisimple  Lie algebra $\k$, {\it the number of simple factors of  $\k$ is denoted by $\# \k$.}

\begin{proposition}\label{coh-G-3}
 Let $G$ be a compact connected Lie group.
 Then 
 $$
 \dim   H^3(G ) \, = \  \# [\g,\, \g] + \dim   \wedge^3 \z(\g)\,.
 $$
\end{proposition}
\begin{proof}
 See \cite[Corollary I, p.\,205]{GHV-3}.  
\end{proof}

\begin{corollary} \label{cor-dim-Pg3}
Let $G$ be a compact connected Lie group, and $P_\g^3$ be the space of primitive three-forms. Then
 \[\dim   P_\g^3 \,=\, \#[\g,\, \g]  \,.\]
\end{corollary}
\begin{proof}
The proof follows from Proposition \ref{coh-G-3} and  Hopf's Theorem  which says that  for any compact connected Lie group $ G $,  the cohomology algebra $ H^*(G,\R ) \simeq \wedge P_\g$.
\end{proof}

Let $G$ be a compact connected Lie group, and $r \,:=\, \#[\g,\, \g]$.  Let 
\begin{align}\label{decomposition-g-simple-factors}
 \g\,=\, \z(\g)\,\oplus\, \g_1\,\oplus\,\cds\,\oplus\, \g_r \,,
\end{align} where $\g_i$'s are simple factors in $[\g,\,\g]$.  
Let ${\Bi}_{\g_i}$ denote the Killing form on $\g_i$, and 
$\pi^\g_{\g_i} : \g \to \g_i$ be the associated projection (see \eqref{notation-projection}).
For simplicity we will denote $\pi^\g_{\g_i}$ by $\pi_i$.  
Then  $\widetilde {\Bi}_{\g_i}^\g$ on $\g$ is defined, as  in \eqref{definition-B}, by
\begin{equation}\label{definition-B-tilde}
   \widetilde {\Bi}_{\g_i}^\g(X,\,Y)\,:=\, {\Bi}_{\g_i}(\pi_i(X),\,\pi_i(Y)) \,.
\end{equation}
As before, for simplicity, we will denote $ \widetilde {\Bi}_{\g_i}^\g$ by $\widetilde {\Bi}_{i}$.
Note that   $\widetilde {\Bi}_i\,\in\, S^2(\g^*)^\g$.  
We next define the map $\rho$ as in the following way.
\begin{align}\label{rho-3}
\rho\,\colon\, S^2(\g^*)^\g
\,\lto\, (\wedge^3 \g^*)^\g \,  \,     ;   \quad
\rho(\eta)(x,\,y,\,z)\,:= \,\eta([x,\,y],\,z).
\end{align}
It is also clear that $  \widetilde {\Bi}_i \not\in {\rm ker}\, \rho$ for all  $1\leq i \leq r$.

\begin{lemma}\label{lemma-Pg3-basis} 
Let $r \,:=\, \#[\g,\, \g]$.  Let $P_\g^3$ be the space of primitive three-forms.
 Then the set  $ \{\rho(\widetilde {\Bi}_i)\, \mid\, 1\leq i \leq r \}$ forms a basis of $ P_\g^3$. 
\end{lemma}

\begin{proof}
As the Killing form ${\Bi}_i$ is non-degenerate on $\g_i$, $\rho(\widetilde {\Bi}_i)\,\neq\, 0$ for $1\leq i \leq r$. In view of \cite[Theorem 
II, p. 244]{GHV-3}, $\rho(\widetilde {\Bi}_i )\,\in\, P_\g^3 $. It also follows from the definitions of $\rho$ and $\widetilde {\Bi}_i $ that 
$\{\rho(\widetilde {\Bi}_1) ,\,\dots,\, \rho(\widetilde {\Bi}_r)\} $ are linearly independent. 
Let $\lambda_i, \, i= 1, \dots, r$ be scalars in $\R$ such that  
$\sum_i 
\lambda_i\rho(\widetilde {\Bi}_i )\, =\,0$. Let $X_i,\,Y_i,\,Z_i\,\in \,\g_i, \, i = 1, \dots, r$ satisfying ${\Bi}_i 
([X_i,\,Y_i],\,Z_i)\,\neq\, 0$. If we 
substitute $X_i,\,Y_i,\,Z_i$ in the above relation, we will get $\lambda_i \,=\,0$
for all $i$.  Now the 
proof follows as $\dim  P_\g^3 \,=\,r$; see Corollary \ref{cor-dim-Pg3}.
\end{proof}

Now we define an explicit transgression which will be useful in the computation of the third and fourth cohomology groups.
Using the basis in the Lemma \ref{lemma-Pg3-basis} above, we first define a linear map
\begin{align}\label{def-tau3}
\widetilde \tau^3\colon P_\g^3 \ \lto\ S^2(\g^*)^\g\,;\quad  \widetilde \tau^3( \rho(\widetilde {\Bi}_i )) \,:=\, \widetilde  {\Bi}_i \,.
\end{align} 
We may choose a transgression $\tau = \oplus_i \tau^i$ such that $\tau^3 := \widetilde \tau^3$. 
As $\{\rho(\widetilde {\Bi}_i ) \,\mid \,1\,\leq\, i\,\leq\,  \#[\g, \,\g] \}$ is a basis of $  P_\g^3$,  we may further define 
$ \nabla^3( 1\ox \rho(\widetilde {\Bi}_i )) \,:=\, \widetilde {\Bi}_i |_{\h\times \h} \ox 1\, .$
 The map $\nabla^3$ is defined as follows: 

$\nabla^3\,\colon\, (\h^*)^H \ox P_\g^1 \ \bigoplus\  1\ox P_\g^3  \ \bigoplus\ 1\ox\wedge^3 P_\g^1\,\, \lto\, \
S^2(\h^*)^H \ox1 \ \bigoplus\ (\h^*)^H\ox \wedge^2P_\g^1 $ is given by 
\begin{align}
	\nabla^3(\psi \ox f)\ &=\  \psi \vee f|_\h \ox 1\,,  \qquad  f\,\in\, P_\g^1,   \nonumber   \\
	\nabla^3( 1\ox \rho(\widetilde {\Bi}_i)) \,&=\, \widetilde {\Bi}_i|_{\h\times \h} \ox 1\,,\quad {\rm \ see\ 
		\eqref{def-tau3}\,\, and\,\, Lemma\,\, \ref{lemma-Pg3-basis}}. \label{def-nabla3}
	\\
	\nabla^3( 1\ox f_1\wedge f_2\wedge f_3 ) \,&=\, f_1|_\h\ox f_2\wedge f_3 \ -\   f_2|_\h\ox f_1\wedge f_3\, + \,  f_3|_\h\ox f_1\wedge f_2  ,\, \, f_i\in P_\g^1.\nonumber
\end{align}
The next result will be used in the proof of Theorem \ref{thm-3rd-coh-general}.

\begin{lemma}\label{lem-ker-nb3-isom-NgH}
	There is an isomorphism between the following vector spaces:
$$ \NC_{\g,\h}\,\simeq\,    {\ker \nb^3}|_{    (\a^*)^H\ox ({\rm ker}{\Re} )^\perp  \bigoplus  \Sb_1 \bigoplus 1\ox P_\g^3 }\,.$$
\end{lemma}

\begin{proof} We begin by fixing some more notations which will be used in this proof 
as well as in the proof of Lemma \ref{ker-nb4-complemntry-part}.

We first use the convention mentioned at \eqref{convension1}  and  \eqref{h decomposition refined} to identify  the dual spaces  
$(\a^H)^* , \, ((H-1) \a)^*, \,  \b^*,  \, [\h,\, \h]^*$ as subspaces of $\h^*$. Under this identification, we have
\begin{align}
	\h^* \,=\, (\a^H)^* \oplus ((H-1) \a)^* \oplus \b^* \oplus [\h,\, \h]^*.\label{h* decomposition refined}
\end{align}

We next fix some more notations, which will
be used here and in the proof of Lemma 
\ref{ker-nb4-complemntry-part}.

\begin{enumerate}[label = {{\bf  D}.\arabic*}]
	
\item~ \label{convensio4}	
Let $r_{\a^H} \,:=\, \dim   \a^H, \, r_\b \,:=\, \dim   \b$.  We fix bases $\{A_1,\dots, A_{r_{\a^H}}\}$  and
$\{D_1,\,\dots,\, D_{r_{\b}}\}$  of $\a^H$ and $\b$, respectively. Let $ \{A^*_1,\,\dots, \,A^*_{r_{\a^H}} \} $ be a basis of $(\a^H)^*$
dual to the basis $\{A_1,\,\dots,\, A_{r_{\a^H}}\}$ and $ \{ D^*_1,\,\dots,\, D^*_{r_{\b}}  \} $ be a basis of $\b^*$ dual to the basis
$\{D_1,\,\dots,\, D_{r_{\b}}\}$. Thus $A^*_p, \,D^*_i \,\in\, \h^*$ for all $ 1\,\leq\, p \,\leq\, r_{\a^H},\, \,  1 \,\leq\, i 
\,\leq\, r_\b$. Moreover, $A^*_p (A_q) \,= \,\delta_{pq}, \, A^*_p ( (H-1) \a + \b + [\h,\, \h])
\,=\, 0$, for all $ 1\,\leq\, p,\,q \,\leq\, r_{\a^H}$ and 	$D^*_i ( D_j) \,=\, \delta_{ij}, \, D^*_i ( \a + [\h,\, \h]) \,=\,0$, for
all $ 1 \,\leq\, i, \,j \,\leq\, r_{\b}$. As $\b^H \,=\, \b$ it should be noted that elements of $(\a^H)^*$ and $\b^*$ remain
fixed under the action of $H$. Thus $H$ fixes 	$A^*_p, D^*_i \,\in\, \h^*$ for all $ 1\,\leq\, p \,\leq\, r_{\a^H}, \, \, 1 \,\leq\, i 
\,\leq\, r_\b$.
	
\item~\label{convension5}
Recall that the restriction map $f \,\longmapsto \,f|_\b$ from $( \ker {\Re})^{\perp} $ to $\b^*$ is a linear isomorphism of vector spaces. Using this map, we choose a basis 	
$\{ \lambda_1, \,\cdots,\, \lambda _{r_{\b}}  \} $ of $( \ker {\Re})^{\perp} $ such that $\lambda_i|_\b \,=\, D^*_i|_\b$ for all $1 \,\leq\, i \,\leq \,r_\b$. 	
Moreover as $\lambda_i \,\in\, P_\g^1$, it follows that $\lambda_i ( \a + [\h, \,\h]) \,=\, 0$ for  all $ 1 \,\leq\, i \,\leq\, r_\b$. Thus $\lambda_i|_\h \,=\,  D^*_i $ for all $1 \,\leq\, i \,\leq\, r_\b$.
\end{enumerate}

Now write any arbitrary element of  $(\a^*)^H\ox ({\rm ker}{\Re} )^\perp  \bigoplus  \Sb_1 \bigoplus 1\ox P_\g^3$ with respect to the bases described above and in Lemma \ref{lemma-Pg3-basis}, say,
$$
\sum_{1\le  i \le r_{\a^H} \atop 1\le j \le r_\b}\alpha_{ij} A^*_i\ox  \lambda_j \,+\,\sum_{ 1\le l\leq k\leq r_\b}\beta_{lk} (D^*_l\ox \lambda_k \,+\,D^*_k \ox \lambda_l  ) \,+\, \sum_{1\le i \le \#[\g,\,\g]} 1\ox \eta_i \rho(\widetilde{\Bi}_i)
.$$
 Then  apply $ \nb^3 $ on it :
\begin{equation}\label{elt-rest-part}
	\begin{split}
		\nb^3   &\,\Big(\sum_{ i,j}\alpha_{ij} A^*_i\ox  \lambda_j \,+\,\sum_{ l\leq k}\beta_{lk} (D^*_l\ox \lambda_k \,+\,D^*_k \ox \lambda_l  ) \,+\, \sum_i 1\ox \eta_i \rho(\widetilde{\Bi}_i)
		\Big)      \\
		&=\,	\big(  \sum_{ i,j}\alpha_{ij} A^*_i\vee D^*_j \,+\,\sum_{  l\leq k}2\beta_{lk} D^*_l \vee D^*_k\, +\, \sum_i \eta_i \widetilde{\Bi}_i|_{\h\times\h}  \big) \ox 1
	\end{split}
\end{equation}
Thus using \eqref{elt-rest-part} it is immediate that  the kernel of $\nb^3$  restricted to the subspace $ (\a^*)^H\ox ({\rm ker}{\Re} )^\perp  \bigoplus  \Sb_1 \bigoplus 1\ox P_\g^3$ 
consists of  elements $\sum_{ i,j}\alpha_{ij} A^*_i\ox  \lambda_j \,+\,\sum_{ l,k}\beta_{lk} (D^*_l\ox \lambda_k \,+\,D^*_k \ox \lambda_l  ) \,+\, \sum_i \eta_i \rho(\widetilde{\Bi}_i)$ such that the symmetric bilinear form 
\begin{align}
\sum_{ i,j}\alpha_{ij} A^*_i\vee D^*_j \,+\,\sum_{  l\leq k}2\beta_{lk} D^*_l \vee D^*_k\, +\, \sum_i \eta_i \widetilde{\Bi}_i|_{\h\times\h}\label{elt-rest-part2}
\end{align}
is identically zero on $\h \times \h$. 

Observe that the form appearing in \eqref{elt-rest-part2} is $H$-invariant. Thus by Proposition \ref{representation-proposition}, the 
subspaces $\h^H$ and $(H-1) \h$ are orthogonal with respect to this form. Furthermore, by Lemma \ref{lem-orthog-ideals}, as $(H-1) \h \,=\, 
(H-1) \z (\h) + [\h,\, \h]$, the spaces $(H-1) \z(\h) \,=\, (H-1) \a$ and $[\h,\, \h]$ are orthogonal with respect to the form in 
\eqref{elt-rest-part2}. Recall that $\h^H \,=\, \a^H + \b$. Thus it follows that the form as in \eqref{elt-rest-part2} is zero on $\h \times \h$ 
if and only if it is zero on each of the spaces $\a^H \times \a^H,\, \a^H \times \b,\, \b \times \b,\, (H-1) \h \times (H-1) \h$.

We next evaluate the form as in  \eqref{elt-rest-part2} at $ (D_j, \,D_j) $ and $(D_i,\,D_j)$ for $i\,\neq\, j$,  respectively:
\begin{align}
	&4\beta_{jj} \,+\,  \sum_l \eta_l \widetilde{\Bi}_l(D_j,\, D_j) \,=\,0,    \label{value-beta-jj}  \\
	&  2 \beta_{ij} \,+\, \sum_l \eta_l \widetilde{\Bi}_l(D_i,\, D_j)\,=\,0     \, , \quad {\rm for\,\,  } i\,\neq\, j \,.   \label{value-beta-ij}
\end{align}
Similarly, by  evaluating  the form as in  \eqref{elt-rest-part2} at $(A_i,\,D_j) $ we have :
\begin{align}\label{value-alpha-ij}
	\alpha_{ij}  \,+\, \sum_l \eta_l \widetilde{\Bi}_l(A_i, \,D_j )\,=\,0   \,.   
\end{align}
It is immediate that the form as in \eqref{elt-rest-part2} is zero on each of the spaces $a^H \times \b, \, \b \times \b$ if and only if 
\eqref{value-beta-jj}, \eqref{value-beta-ij} and \eqref{value-alpha-ij} hold. Moreover, the form as in \eqref{elt-rest-part2} is zero on each of the
spaces $\a^H \times \a^H, (H-1) \h \times (H-1) \h$ if and only if 
\begin{align}
\sum_i \eta_i \widetilde{\Bi}_i|_{ \big( \a^H + (H-1) \h \big) \times \big( \a^H + (H-1) \h \big)} = 0.\label{restriction-killing-form-condition}
\end{align}
In view of \eqref{value-beta-ij}, \eqref{value-beta-jj}, \eqref{value-alpha-ij}, \eqref{restriction-killing-form-condition} and Lemma 
\ref{lem-equivalence-G-inv-forms} we conclude that $ \NC_{\g,\h}$ and $ {\ker \nb^3}$ restricted to the subspace ${ (\a^*)^H\ox ({\rm ker}{\Re})^\perp \bigoplus \Sb_1 \bigoplus 1\ox P_\g^3 }$ are isomorphic.
\end{proof}

\medskip
{\bf Proof of Theorem \ref{thm-3rd-coh-general}.}
Recall the Koszul complex $({\CS}^*:=S(\h^*)^H\ox\wedge P_\g,\, \nabla )$ as defined in Section \ref{sec-Koszul-complex}. As $G$ is compact connected, the cohomology of $G/H$ is given by the cohomology of the Koszul complex  $(S(\h^*)^H\ox\wedge P_\g,\, \nabla )$; see    Theorem \ref{Cartan-thm-disconnected}. Therefore, 
\begin{align}\label{h3}
H^3(G/H)\,\simeq\, \frac{{\rm ker}\, \nabla^3 }{{\rm Im}\, \nabla^2 }\,,
\end{align}
where the differentials $\nabla^3$ and $\nabla^2$ are defined in \eqref{def-nabla3} and \eqref{def-nabla2}, respectively. The numerator and the denominator in \eqref{h3} will be identified.   Recall that 
$$
{\CS}^3\, = \, (\h^*)^H\ox P_\g^1 \ \bigoplus\  1\ox P_\g^3  \, \bigoplus\ 1\ox\wedge^3 P_\g^1\,,
$$
and  $\Im \nb^2 \,\subset\, (\h^*)^H\ox P_\g^1 $. In fact, ${\Im \nb^2} \bigoplus (\a^*)^H\ox  P_\g^1 \bigoplus \Sb_1 =   (\h^*)^H \ox P_\g^1\,$; see Lemma \ref{lem-ker-im-nb2}(5). Also recall that $\nb^2({(\a^*)^H\ox  \ker {\Re}}  ) =0 $.
Now 
\begin{align*}\label{comp}
\frac{{\rm ker}\, \nabla^3 }{{\rm Im}\, \nabla^2 }\,&\simeq \, \frac{{\ker} \nb^3|_{\big(   (\h^*)^H\ox P_\g^1 \big)  \bigoplus 1\ox P_\g^3   }}{(\Im \nb^2)}\, \bigoplus\, {\rm ker}\,\nabla^3|_{1\ox \wedge^3 P_\g^1} \\
&\, \simeq \,  {\ker}\nb^3|_{{(\a^*)^H\ox  \ker {\Re}}\bigoplus (\a^*)^H\ox ({\rm ker}{\Re} )^\perp  \bigoplus  \Sb_1 \bigoplus 1\ox P_\g^3  }  \,\bigoplus\, 1\ox \wedge^3\, {\rm ker}{{\Re}}\\		
&\, \simeq \,   {(\a^*)^H\ox  \ker {\Re}}  \bigoplus   {\ker}\nb^3|_{   (\a^*)^H\ox ({\rm ker}{\Re} )^\perp  \bigoplus  \Sb_1 \bigoplus 1\ox P_\g^3  }  \,\bigoplus\, 1\ox \wedge^3\, {\rm ker}{{\Re}}\\
		&\, \simeq \, \big((\z(\h)\cap [\h,\h])^*\big)^H\ox \Big(\frac{\g}{[\g,\,\g]+\h}\Big)^* \   \bigoplus  \NC_{\g,\h}    \,\bigoplus\, \wedge^3 \Big(\frac{\g}{[\g,\,\g]+\h}\Big)^*\,.
	\end{align*}
The last isomorphism follows from \eqref{ker-nbl},     Lemma \ref{lem-ker-nb3-isom-NgH},   and Lemma   \ref{lem-ker-im-nb2}(1).

If furthermore,  $\z(\h) \cap [\g,\,\g])=0$ then 
$H^3(G/H)\simeq   \NC_{\g,\h}\, \bigoplus\, \wedge^3 {\Big(\frac{\g}{[\g,\,\g]+\h}\Big)}^*\simeq H^3(G/H^0) $\,. Thus $(1)$ follows.  The statement $(2)$ follows from the additional fact that  $\g = [\g,\g]+\h$.
\qed

\subsection{Description of fourth cohomology}
Here we will prove Theorem \ref{thm-4th-coh-general}.  We follow the strategy as mentioned in Section \ref{strategy}.

We begin by  decomposing $(\h^*)^H\ox( \wedge^2 P_\g^1)$ into various subspaces using \eqref{h decomposition component2} and \eqref{decomposition-Pg1-ker-Rgh-ker-perp}  as follows :
\begin{equation}
 \begin{split}\label{decomp-h-ox-wd2Pg}
(\h^*)^H & \ox \wedge^2 P_\g^1\,\,    = \,\,    
(\a^*)^H\ox   \big(  \wedge ^2 P_\g^1 \big)    \bigoplus\, \b^*\ox \big(  \wedge ^2 P_\g^1 \big)   \\
  \,=&\,  (\a^*)^H\ox \wedge^2 {\ker} {\Re}\, \bigoplus\, (\a^*)^H\ox \big(({\ker} {\Re})^\perp  \wedge {\ker} {\Re} \big) \, \bigoplus\, (\a^*)^H\ox \wedge^2 ({\ker} {\Re})^\perp\, \\
& \bigoplus  \,	\b^*\ox \wedge^2 {\ker} {\Re} \,\bigoplus \,\b^*\ox \big(({\ker} {\Re})^\perp  \wedge {\ker} {\Re}  \big)\, \bigoplus\, \b^*\ox \wedge^2 ({\ker} {\Re})^\perp \,. 
 \end{split}
\end{equation}
We will next define two subspaces $\mathbf{A_2} $ and $\mathbf{S_2} $ of $\b^*\ox  \big( {\ker {\Re}} \wedge ({\ker {\Re}})^\perp\big) $ as follows :
\begin{align}
\mathbf{A_2}&:= \text{Span} \{ f_1|_\b \ox (g\wedge f_2 )-f_2|_\b \ox (g\wedge f_1 )\,\mid\, f_1,\,f_2\,\in\, ({\ker {\Re}})^\perp, g\in {\ker}\Re\}, \label{def-A2}     \\
	\mathbf{S_2}&:= \text{Span}\{ f_1|_\b \ox ( g\wedge f_2 )  + f_2|_\b \ox  ( g\wedge f_1 )\, \mid\,  f_1,f_2\in ({\ker {\Re}})^\perp, g\in {\ker}\Re\}. \label{def-S2} 
\end{align}
Similarly, we  define another   two subspaces $\mathbf{A_3} $ and $ \mathbf{S_3} $ of 
$  \b^*\ox  \big( ({\ker \Re})^\perp \wedge ({\ker {\Re}})^\perp \big) $ as follows:
\begin{align}
\mathbf{A_3}&:= \text{Span}\Bigg\{\begin{array}{cc}
f_1|_\b \ox ( f_2\wedge f_3)-f_2|_\b \ox(f_1\wedge f_3) \\
\qquad  + \,f_3|_\b\ox (f_1\wedge f_2)
\end{array} 
\biggm|   f_1, f_2,f_3\in ({\ker {\Re}})^\perp \Bigg\},  \label{def-A3} \\	 
\mathbf{S_3}&:= \text{Span} \Bigg\{\begin{array}{cc}
	-2f_1|_\b \ox ( f_2\wedge f_3) -  f_2|_\b \ox(f_1\wedge f_3) \\
	\qquad  +\, f_3|_\b\ox (f_1\wedge f_2)
\end{array} 
\biggm|   f_1, f_2,f_3\in ({\ker {\Re}})^\perp \Bigg\}\,.\label{def-S3} 
\end{align}

\begin{lemma}\label{lem-im-nb3-complementary-part}
Let $\Ab_2, \,\Sb_2,\, \Ab_3$ and $\Sb_3$ be defined as above. Then the following statements hold: 
\begin{enumerate}
\item $\nb^3\big( 1\ox ({\ker } {\Re} \wedge  {\ker } {\Re}\wedge  ({\ker } {\Re})^{\perp})\,  \big)  \,=\, \b^*\ox( \wedge^2 {\ker} {\Re}) $.\vspace*{2mm}
	
\item $	\nb^3\big( 1\ox ({\ker } {\Re} \wedge  ({\ker } {\Re})^{\perp}\wedge  ({\ker } {\Re})^{\perp}) \big)  \, =\, \Ab_2 $  and   \vspace*{1mm}  \\
$\nb^3\big( 1\ox (({\ker } {\Re})^{\perp} \wedge  ({\ker } {\Re})^{\perp}\wedge  ({\ker } {\Re})^{\perp}) \big) \,=\, \Ab_3 $. \vspace*{2mm}
\item $ \nb^3( 1\ox \wedge ^3 P^1_\g)\,=\, \, \b^*\ox \wedge^2 {\ker} {\Re} \, \bigoplus\,  \mathbf{A_2} \, \bigoplus\, \mathbf{A_3}$. \vspace*{2mm}
\item  $ \Ab_2\oplus\Sb_2 \,=\, \b^*\ox   ( {\ker} {\Re} \wedge ({\ker } {\Re})^{\perp})$.   	\vspace*{2mm}
\item  $ \Ab_3\oplus\Sb_3\,=\, \b^*\ox   (( {\ker} {\Re} )^{\perp}\wedge ({\ker } {\Re})^{\perp})$.   	\vspace*{2mm}
\item  $ \nb^3( 1\ox \wedge ^3 P^1_\g)  \, \bigoplus\,  \mathbf{S_2} \, \bigoplus\, \mathbf{S_3}    \, \bigoplus\, (\a^*)^H\ox \big(  \wedge ^2 P_\g^1 \big)   \,=\,   (\h^*)^H\ox \big(\wedge ^2 P_\g^1 \big)$.
\end{enumerate}
\end{lemma}     
\begin{proof}
Proof of (1) follows from the fact that $ {\Re}\colon ({\ker }{\Re}) ^\perp \,\lto \,\b^*$ is a linear  isomorphism. The statement in (2) is immediate from the definition of the map $ \nb^3 $. Now (3) follows from (1) and (2).

To proof of (4) and (5),  
define an operator $T$ on the vector space $  \ker \Re^\perp \ox \big(( \ker \Re^\perp \oplus  \ker \Re)\wedge ( \ker \Re^\perp \oplus  \ker \Re)\big)$ as follows :
\begin{align*}\label{mapp-for-E0E2}
	T\,\colon\,  \ker \Re^\perp \ox \big(( \ker \Re^\perp \oplus  \ker \Re)& \wedge   ( \ker \Re^\perp \oplus  \ker \Re)\big) \, \lto \,  \\
	& \ker \Re^\perp \ox \big(( \ker \Re^\perp \oplus  \ker \Re)\wedge ( \ker \Re^\perp \oplus  \ker \Re)\big) \\
	v\ox \big((v_1+w_1)\wedge (v_2+w_2)\big)\, \,\longmapsto\,\,  & v\ox \big((v_1+w_1)\wedge (v_2+w_2)\big)- v_1\ox \big(v\wedge (v_2+w_2)\big) \nonumber \\
	&\qquad  + v_2\ox \big(v\wedge (v_1+w_1)\big)\,\nonumber .
\end{align*}
Note that $ \ker \Re^\perp \ox \big(  \ker \Re\wedge  \ker \Re^\perp \big) $ and  $ \ker \Re^\perp \ox \big(  \ker \Re^\perp \wedge  \ker \Re^\perp \big) $  are $ T $-invariant. Let $ T_1\,:=\, T|_{ \ker \Re^\perp \ox (  \ker \Re\wedge  \ker \Re^\perp )} $, and $ T_2\,:=\, T|_{ \ker \Re^\perp \ox (  \ker \Re^\perp \wedge  \ker \Re^\perp )} $. Now $T_1^2 - 2T_1=0$ and $T_2^2- 3T_2=0$.   
Let $E_\lambda(T_i)$ be the eigenspace of $E_i$ for the eigenvalue $\lambda$.
Then  we have the following direct sum decomposition: 
\begin{equation}\label{decomp-vv2030}
	\begin{split}
& \ker \Re^\perp \ox (  \ker \Re\wedge  \ker \Re^\perp ) \,=\, E_2(T_1)\,\oplus\,E_0(T_1)\\
 &  \ker \Re^\perp \ox ( \ker \Re^\perp \wedge  \ker \Re^\perp )\,=\, E_3(T_2)\oplus E_0(T_2)\,.
 \end{split}
\end{equation}
Now the proof of (4) and (5) follows form a minor variant of \eqref{decomp-vv2030}.  Finally, (6) follows from (3), (4), (5) and \eqref{decomp-h-ox-wd2Pg}.
\end{proof}

The differential $ \nb^4 $ is given as below:  
\begin{align*}
	\nb^4\,\colon\, \big(S^2(\h^*)^H\ox 1  \,\bigoplus\,& (\h^*)^H\ox \wedge^2P_\g^1\, \bigoplus\, 1\ox (P_\g^3\wedge P_\g^1)\, \bigoplus\, 1 \ox \wedge^4 P_\g^1  \big) \,  \vspace{0.2 cm}\\
	& \lto\,\, S^2(\h^*)^H\ox P_\g^1\, \,\bigoplus\, \,(\h^*)^H\ox P_\g^3\, \,\bigoplus\, \, (\h^*)^H\ox \wedge^3P_\g^1 	
\end{align*} 
\begin{equation}\label{def-nabla4}
	\begin{split}
		\nabla^4( S^2(\h^*)^H\ox 1)  \,=\,&  \,0 \,,\\
		\nabla^4(\psi\ox f_1\wedge f_2)  \,=\,& \, \psi \vee f_1|_{\h} \ox f_2 - \psi \vee f_2|_{\h} \ox f_1 \,,\qquad  f_1,
		\,f_2 \,\in \,P_\g^1\,,\\
		\nabla^4(1\ox \rho(\widetilde {\Bi}_i)\wedge f)  \,=\,& \, \widetilde {\Bi}_i|_{\h\times \h} \ox f  \, -  \, f|_\h\ox \rho(\widetilde {\Bi}_i)\,,\\
		\nabla^4(1\ox f_1\wedge f_2\wedge f_3\wedge f_4 ) \,=\,&\,   f_1|_{\h} \ox f_2\wedge f_3\wedge f_4 -f_2|_\h\ox f_1 \wedge f_3\wedge f_4  \\
		+\,&f_3|_{\h} \ox f_1\wedge f_2\wedge f_4-f_4|_{\h} \ox f_1\wedge f_2\wedge f_3     ,\quad f_i\,\in\, P_\g^1.
	\end{split}
\end{equation}

Recall that in view of Lemma \ref{lem-im-nb3-complementary-part}(6), the following isomorphism in \eqref{ker-nb4-middle-part} holds.  Since $ \nb^4\big((\a^*)^H\ox   ({\ker}   {{\Re}}  \wedge {\ker}   {{\Re}}) \big)  = 0$,
the isomorphism in \eqref{ker-nb4-middle-part-refine} follows from the direct sum decomposition  in \eqref{decomp-h-ox-wd2Pg}.
\begin{align}
\frac{{\ker}\nb^4|_{(\h^*)^H\ox \big(  \wedge ^2 P_\g^1 \big) \bigoplus 1\ox (P_\g^3\wedge P_\g^1 )    } }{\nb^3( 1\ox \wedge ^3 P^1_\g) }
&\,\,\simeq\, {\ker}\nb^4|_{ \mathbf{S_2} \, \bigoplus\, \mathbf{S_3}    \, \bigoplus\, (\a^*)^H\ox \big(  \wedge ^2 P_\g^1 \big)   \bigoplus 1\ox (P_\g^3\wedge P_\g^1 )    }    \label{ker-nb4-middle-part}\\
&\,\,\simeq (\a^*)^H\ox({\ker}   {{\Re}}  \wedge {\ker}   {{\Re}}) \, \, \bigoplus   \label{ker-nb4-middle-part-refine} \\
   {\ker \nb^4}|&_{(\a^*)^H\ox  (({\ker}   {{\Re}})^\perp  \wedge {\ker}   {{\Re}} )  \bigoplus  (\a^*)^H\ox \wedge^2 ({\ker}   {{\Re}})^\perp     \bigoplus   \Sb_2 \bigoplus 
	\Sb_3 \bigoplus 1\ox(P_\g^3\wedge P_\g^1)  }.\nonumber
\end{align}
 
For  a vector space $ V $  and  a positive integer $ k $,
$V^k \,:=\,\underset{k\text{-many}}{\underbrace{V \oplus \cdots \oplus V}}\, .$

\begin{lemma}\label{ker-nb4-complemntry-part}
The following isomorphisms hold:	
	\begin{align}
		{\ker }\nb^4 &  |_{   (\a^*)^H\ox   \big(({\ker}   {{\Re}})^\perp  \wedge {\ker}   {{\Re}} \big)  \bigoplus  (\a^*)^H\ox \wedge^2 ({\ker}   {{\Re}})^\perp \bigoplus   \Sb_2 \bigoplus 
			\Sb_3  \bigoplus     1\ox  (P_\g^3 \wedge P_\g^1)  } \nonumber \\
		\simeq& \,  \,\,{\ker }\nb^4|_{(\a^*)^H\ox   \big(({\ker}   {{\Re}})^\perp  \wedge {\ker}   {{\Re}} \big)  \bigoplus   \Sb_2 \bigoplus    1\ox  (P_\g^3 \wedge {\ker }  {{\Re}})} \nonumber  \\
		& \quad  \bigoplus\, {\ker }\nb^4   |_{    (\a^*)^H\ox \wedge^2 ({\ker}   {{\Re}})^\perp \bigoplus  \Sb_3  \bigoplus     1\ox  (P_\g^3 \wedge ({\ker}   {{\Re}})^\perp  )}  \label{1st-isom-lem-ker0}   \\
	\simeq & \,\,  \, ( \NC_{\g,\h})^{\dim   {\ker   {{\Re}}}}   \label{2nd-isom-lem-ker0}  \,.
	\end{align}
\end{lemma}                 

\begin{proof}
Observe that 
\begin{align*}
(\a^*)^H &\ox   \big(({\ker}   {{\Re}})^\perp  \wedge {\ker}   {{\Re}} \big)  \bigoplus  (\a^*)^H\ox \wedge^2 ({\ker}   {{\Re}})^\perp \bigoplus   \Sb_2 \bigoplus \Sb_3  \bigoplus     1\ox  (P_\g^3 \wedge P_\g^1)    \\  
\,=&\,\Big( (\a^*)^H\ox   \big(({\ker}   {{\Re}})^\perp  \wedge {\ker}   {{\Re}} \big)  \bigoplus   \Sb_2 \bigoplus    1\ox  (P_\g^3 \wedge {\ker }  {{\Re}})\,\Big)\\
& ~~ \, \bigoplus\, \Big ((\a^*)^H\ox \wedge^2 ({\ker}   {{\Re}})^\perp \bigoplus  \Sb_3  \bigoplus     1\ox  (P_\g^3 \wedge ({\ker}   {{\Re}})^\perp )\Big)\,.
\end{align*}
In view of \eqref{def-nabla4}, it follows that
   \begin{align} 
   	\nb^4\Big(& (\a^*)^H \ox \big({\ker}   {{\Re}}\wedge ({\ker}{{\Re}})^\perp \big) \,\bigoplus\, \mathbf{S_2}  \,\bigoplus\,  1\ox \big(P_\g^3 \wedge {\ker}   {{\Re}}\big)\,     \Big)\nonumber\\
   	& \qquad \subseteq \, S^2(\h^*)^H\ox {\ker}  {{\Re}} \,, \  {\rm and}     \label{inclusion-I}\\
	\nb^4\Big(  &  (\a^*)^H\ox \big(({\ker}   {{\Re}})^\perp  \wedge ({\ker}   {{\Re}})^\perp \big) \,\bigoplus\, \mathbf{S_3}  \,\bigoplus\,  1\ox \big(P_\g^3 \wedge ({\ker}   {{\Re}})^\perp \big)\,     \Big)\,\nonumber\\
	& \qquad  \subseteq \, S^2(\h^*)^H\ox ({\ker}  {{\Re}} )^\perp  \,\bigoplus \, \b^*\ox  P_\g^3\,.   \label{inclusion-II}
\end{align}
Since the right hand sides of \eqref{inclusion-I} and \eqref{inclusion-II} are distinct direct summands of $S^2(\h^*)^H\ox P_\g^1\, \bigoplus\, 
\b^*\ox P_\g^3$, \eqref{1st-isom-lem-ker0} holds.

Let  $r_0 \,:=\,  \dim   {\ker} {\Re}\,=\,\dim   \frac{\g}{[\g,\,\g]+\h} $.  Let $ \{ \nu_1,\,\cdots,\, \nu_{r_0} \} $ be  a basis of $\ker {\Re} $.
We now adhere to the notations used in \eqref{convensio4} and \eqref{convension5}. Let  $\Ab_2,\, \Sb_2, \,\Ab_3$ and $\Sb_3$ be defined as in \eqref{def-A2}, \eqref{def-S2},\eqref{def-A3} and \eqref{def-S3}, respectively.
Then observe that 
$ \{ D^*_i \ox (  \nu_k\wedge \lambda_l )   \,-\, D^*_l \ox  ( \nu_k\wedge \lambda_i )\, \mid\,    1\leq i< l \leq  r_{\b}, \, 1\leq  k \leq r_0 \}$ and $\{D^*_i \ox  ( \nu_k\wedge  \lambda_l )\,+\,  D^*_l \ox  ( \nu_k\wedge \lambda_i ) \,\mid \,  1\leq i \leq l \leq  r_{\b}, \, 1\leq   k \leq r_0 \}$ form a basis of $\mathbf{A_2}$ and $\mathbf{S_2}$, respectively.  Similarly,  
$\{ D^*_i \ox  ( \lambda_k\wedge \lambda_l )  -D^*_k \ox  ( \lambda_i\wedge  \lambda_l ) +  D^*_l \ox  ( \nu_k\wedge \lambda_i ) \,\mid\,    1\,\leq\, i  \,<\, k\,<\,l \,\leq
\, r_{\b} \}$, 
$\, \{ -2D^*_i \ox ( \lambda_k\wedge \lambda_l )  -D^*_k \ox     (\lambda_i\wedge \lambda_l )  + D^*_l \ox (  \lambda_i\wedge \lambda_k )\,\mid\, 1\,\leq k \,\leq\, i\, \leq\,  r_{\b},\, k\,<\,l\leq r_{\b} \}$ 
form a basis of $\mathbf{A_3},\, \mathbf{S_3}$, respectively.

We divide the rest of  the proof into two steps. 

{\it Step 1.}  In this step we will prove  
\begin{align}\label{step-1-ker-nb4-0}
{\ker }\nb^4|_{(\a^*)^H\ox( \wedge^2 ({\ker}   {{\Re}})^\perp) \bigoplus \Sb_3  \bigoplus 1\ox  (P_\g^3 \wedge ({\ker}{{\Re}})^\perp)} \,=\,0\,. 
 \end{align}

We begin by applying $ \nb^4 $ on any arbitrary element of $(\a^*)^H\ox(\wedge^2 ({\ker}{{\Re}})^\perp) \bigoplus \Sb_3  \bigoplus 1\ox  (P_\g^3 \wedge ({\ker}{{\Re}})^\perp)$,
say, 
\begin{align}\label{general-element}
  & \sum_{1\leq i \leq r_{\a^H}\atop 1\leq j<l \leq r_\b}  \mu_{ijl} A^*_i\ox  (\lambda_j\wedge \lambda_l )\, \bigoplus\,  
\,	\sum_{1\le j\leq i \le r_\b \atop  j<l \leq r_\b}  \alpha_{ijl} \Big(  -2D^*_i \ox  (\lambda_j\wedge \lambda_l )  -D^*_j \ox  ( \lambda_i\wedge \lambda_l)   \nonumber \\	
&	      +    D^*_l \ox  (\lambda_i\wedge \lambda_j)\,   \Big)
\, \bigoplus\, \sum_{1\le  i \le \#[\g,\,\g] \atop 1\le j \le r_\b}1\ox \zeta_{ij} (\rho(\widetilde{\Bi}_i)\wedge \lambda_j) 
\end{align}
and we obtain
\begin{align}\label{eq-nb4-0}
	\nb^4\Big(    & \sum_{i\atop j<l}  \mu_{ijl} A^*_i\ox  (\lambda_j\wedge \lambda_l )\, \bigoplus\,  
	\,	\sum_{ j\leq i \atop  j<l}  \alpha_{ijl} \Big(  -2D^*_i \ox  (\lambda_j\wedge \lambda_l )  -D^*_j \ox  ( \lambda_i\wedge \lambda_l)   \nonumber \\	
	&	      +    D^*_l \ox  (\lambda_i\wedge \lambda_j)\,   \Big)
\, \bigoplus\, \sum_{i,j}1\ox \zeta_{ij} (\rho(\widetilde{\Bi}_i)\wedge \lambda_j) \Big) \,\,\,=\,\,    \\
	 \sum_{i\atop  j<l} & \mu_{ijl}  \big( (A^*_i\vee  D^*_j)\ox \lambda_l \, -\,  (A^*_i\vee  D^*_l )\ox \lambda_j \big) \bigoplus\,  	\sum_{j\leq i\atop  j<l}3\alpha_{ijl} \big(  (  D^*_i\vee  D^*_l)\ox  \lambda_j     \nonumber \\	
	& -   (D^*_i\vee   D^*_j)\ox  \lambda_l  \big)   \,+\, \sum_{i,j}\zeta_{ij} \big( \widetilde{\Bi}_i|_{\h\times\h}\ox  \lambda_j  -  D^*_j\ox \rho(\widetilde{\Bi}_i) \big).  \nonumber
\end{align}	

We will now equate the  right hand side of \eqref{eq-nb4-0} to zero, and show that this will
force all the coefficients   $ \mu_{ijl},
\alpha_{ijl}, \zeta_{ij}$ in \eqref{general-element}  to be zero. 

Since $D^*_j\ox \rho(\widetilde{\Bi}_i)  $	is a part of   basis of  $ (\h^*)^H\ox P_\g^3 $, it follows that $ \zeta_{ij}=0 $ for all $ i,j $.

Next consider the term of the form $ (-\vee -) \ox  \lambda_{r_\b} $ in \eqref{eq-nb4-0}. Then
\begin{align*}
\Big(  \sum_{i\atop  j<r_\b} \mu_{ijr_\b}   (A^*_i\vee  D^*_j) \,- \,\sum_{j\leq i\atop  j<r_\b}3\alpha_{ijr_\b}  (  D^*_i\vee  D^*_j)  \Big)\ox  \lambda_{r_\b}   	\,=\,0\,.
\end{align*}
From the above equation  it follows by substituting at dual basis element that  $ \mu_{ijr_\b} =0  $ and  $ \alpha_{ijr_\b} =0  $.  Next  we will consider the  term of the form $ (- \vee -) \ox \lambda_{r_\b-1} $ in \eqref{eq-nb4-0}.  We will get the following relation:
\begin{align*}
	\Big(  \sum_{i\atop  j<r_\b-1} \mu_{ijr_\b}   (A^*_i\vee  D^*_j) \,- \,\sum_{j\leq i\atop  j<r_\b-1}3\alpha_{ijr_\b}  (  D^*_i\vee  D^*_j)  \Big)\ox  \lambda_{r_\b-1}   	\,=\,0\,.
\end{align*}
Similarly, we conclude that   $ \mu_{ijr_\b-1} \,=\,0$ and  $\alpha_{ijr_\b-1} \,=\,0$. 
Proceeding in this way we can conclude that all the coefficients in \eqref{general-element} are $ 0 $. This proves \eqref{step-1-ker-nb4-0}.

{\it Step 2.} In this step we will prove the following isomorphism:
\begin{align}\label{step-2-ker-nb4-NgH}
 {\ker }\nb^4|_{(\a^*)^H\ox\big({\ker}{{\Re}}\wedge ({\ker}{\Re})^\perp \big) \bigoplus \Sb_2 \bigoplus  1\ox  (P_\g^3 \wedge {\ker}{{\Re}})}\,\simeq\, (\NC_{\g,\h})^{\dim   {\ker   {{\Re}}}} \,.
\end{align}

As in Step 1 we will apply the map  $ \nb^4 $ on any arbitrary element of  $(\a^*)^H\ox\big({\ker}   {{\Re}}  \wedge ({\ker}   {{\Re}})^\perp \big)  \bigoplus   \Sb_2$ $ \bigoplus    1\ox  (P_\g^3 \wedge {\ker }  {{\Re}})$. 
\begin{align}\label{nb4-I-III-Va}
\nb^4\Big( 	\sum_{i,j,l}  \mu_{ijl} A^*_i\ox ( \nu_l \wedge \lambda_j      ) \, \,+\, \,  &	\sum_{l \atop i\leq j}\beta_{ijl} \big(   D^*_i\ox (\nu_l\wedge \lambda_j )+  D^*_j\ox ( \nu_l\wedge \lambda_i) \,  \big) \, \,+\, \,     \nonumber \\	
\sum_{i,j}  1\ox \xi_{ij}(\rho(\widetilde{\Bi}_i)\wedge \nu_j)   \Big) 	\,\,
=&\,\,	\sum_{i,j,l}-\mu_{ijl} (A^*_i\vee  D^*_j  )\ox  \nu_l \,\\
\,&\,+\, \, \,  \sum_{ l \atop i\leq j }  -2\beta_{ijl}   (D^*_i \vee D^*_j )\ox \nu_l   
\, \,+\,  \, \sum_{i,j}\xi_{ij} \widetilde{\Bi}_i|_{\h\times\h}\ox  \nu_j  \nonumber	\\
=\, \, \sum_k \Big(  \sum_{i,l}-\mu_{ilk} (A^*_i\vee  D^*_l  )\,
&\,\,+\,   \sum_{i \leq  j}  -2\beta_{ijk}   (D^*_i \vee D^*_j )
\, \,+\,  \, 	 \sum_i\xi_{ik} \widetilde{\Bi}_i|_{\h\times\h}  \Big)\ox \nu_k    \nonumber 
\end{align}
Since $ \{ \nu_k\,\mid\,  1\,\le\, k\,\le\, r_0\} $ is a basis of $ {\ker}{{\Re}} $, it follows that the left hand side of \eqref{nb4-I-III-Va} is zero if and only if $  \sum_{i,l}-\mu_{ilk} (A^*_i\vee  D^*_l )\,
+\,   \sum_{i \leq j}  2\beta_{ijk}   (D^*_i\vee D^*_j)\,+\, \sum_i\xi_{ik} \widetilde{\Bi}_i|_{\h\times\h}\,=\,0$, for all $k\,=\,1,\,\dots,\, r_0$. 
Now the isomorphism \eqref{step-2-ker-nb4-NgH}  follows from  \eqref{elt-rest-part} and  Lemma \ref{lem-ker-nb3-isom-NgH}.

The isomorphism in \eqref{2nd-isom-lem-ker0} follows from  \eqref{step-1-ker-nb4-0} and \eqref{step-2-ker-nb4-NgH}. This completes the proof.
\end{proof}

\begin{lemma}\label{cokernel-part-in 4coh} Let $ \CC_{\g,H} $ be as in Definition \ref{def-NC-gH}. Then
there is an isomorphism between the following two vector spaces:
\begin{align*}
\frac{S^2(\h^*)^H \ox 1 }{\nabla^3\big((\h^*)^H\ox P_\g^1 {\mathlarger {\mathlarger{\mathlarger{\oplus}}}} \,\,1\ox P_\g^3 \big) }\, \simeq \, \CC_{\g,H}\,.
\end{align*}
\end{lemma}

\begin{proof}
Recall that  $  \a\,=\, \z(\h)\cap [\g,\,\g] ,\, \h^H \,=\,\a^H\oplus\b,\,  \h\cap [\g,\,\g]\,=\, \a+[\h,\,\h], $  and $ (H-1)\h\,= \, (H-1)\a + [\h,\,\h]$; see \eqref{h decomposition component1},  \eqref{h decomposition component2}. Now using Proposition  \ref{representation-proposition} and Lemma \ref{lem-orthog-ideals}, we have the
following decomposition:
	\begin{align*}
		S^2(\h^*)^H \, &=\, S^2\big((\h^H)^* \big)^H \,\oplus\,  S^2\big(((H-1)\h)^* \big)^H     \\
		&=\, S^2((\h^H)^* ) \,\oplus\,  S^2\big(((H-1)\h)^* \big)^H     \\
		&=\,  S^2((\a^H)^*)  \,\oplus\, S^2(\b^*)  \,\oplus\, (\a^H)^*\vee \b^* \,\oplus\,  S^2\big(((H-1)\a)^* \big)^H   \oplus\,  S^2([\h,\h]^* )^H       \\
		&=\, S^2\big((\h \cap [\g,\,\g] )^*\big)^H  \,\oplus\, S^2(\b^*)  \,\oplus\, (\a^H)^*\vee \b^* \,.
	\end{align*}	
	
Let $ r\,:=\,\#[\g,\,\g] $, $ \widetilde{\Bi}_i $ be as in \eqref{definition-B-tilde}, and $ \nb^3 $ be as in \eqref{def-nabla3}. Since $ \a^H, (H-1)\a, [\h,\h] $ are mutually orthogonal with respect to $ \widetilde{\Bi}_i $ for all $ i=1,\dots,r $ , and $ {\Re }(P_\g^1) = \b^*$,  we also have	
	\begin{align*}
		\nb^3 &\big((\h^*)^H\ox P_\g^1 \bigoplus 1\ox P_\g^3     \big) \,=\\ 
		&\Big( (\a^H)^*\vee \b^*\oplus\, S^2(\b^*)  \,\oplus\,     \big\{\sum_i  \eta_i \widetilde {\Bi}_i|_{(\h \cap [\g,\,\g])\times (\h \cap [\g\g] ) }
		\,\mid\, (\eta_1,\, \dots,\, \eta_r)\,\in\, \R^r \big\}              \Big)\ox 1\,.
	\end{align*}	
	Now the proof follows from the fact that  $ \big\{\sum_i  \eta_i \widetilde {\Bi}_i|_{(\h \cap [\g,\,\g])\times (\h \cap [\g\g] ) }
	\,\mid\, (\eta_1,\, \dots,\, \eta_r)\,\in\, \R^r \big\}  \,\subset\, S^2\big((\h \cap [\g,\,\g] )^*\big)^H $.
\end{proof}

\medskip
{\bf Proof of Theorem \ref{thm-4th-coh-general}.} 
Recall the Koszul complex $(\CS^*\,:=\,S(\h^*)^H\ox\wedge P_\g,\, \nabla )$ as defined in Section \ref{sec-Koszul-complex}. 
Using   Theorem \ref{Cartan-thm-disconnected}, we have the following isomorphism:
\begin{align}\label{h4-isom}
	H^4(G/H,\, \R)\, \simeq\, \frac{{\rm ker}\, \nabla^4 }{{\rm Im}\, \nabla^3 }\,,
\end{align}
where the differentials $\nabla^4$ and $\nabla^3$ are as in \eqref{def-nabla4} and \eqref{def-nabla3}, respectively. The numerator and the denominator in \eqref{h4-isom} will be identified.   Recall that 
\begin{align*}
{\CS}^4 \, =\, S^2(\h^*)^H \ox 1 \, \bigoplus\, (\h^*)^H\ox \wedge^2P_\g^1 &  \, \bigoplus\, 1\ox P_\g^3 \wedge P_\g^1 \ \bigoplus\ 1\ox\wedge^4 P_\g^1\,,
\end{align*}
$$\nb^3\big((\h^*)^H\ox P^1_\g \bigoplus 1\ox P_\g^3    \big) \, \subset \, S^2(\h^*)^H \ox 1  \,,\text{ and }  \,  
\nb^3\big(1\ox( \wedge^3  P^1_\g )   \big) \, \subset \, (\h^*)^H \ox( \wedge^2  P^1_\g )    \,.
$$
Thus, the right hand side of \eqref{h4-isom} is isomorphic to the following vector space:
\begin{align*}
 \frac{S^2(\h^*)^H \ox 1 }{\nabla^3\big((\h^*)^H\ox P_\g^1 \, \,{\mathlarger {\mathlarger{\mathlarger{\oplus}}}} \,\,1\ox P_\g^3 \big) }\bigoplus 	\frac{{\ker}\nb^4|_{(\h^*)^H\ox \big(  \wedge ^2 P_\g^1 \big) \bigoplus 1\ox (P_\g^3\wedge P_\g^1 )    } }{\nb^3( 1\ox \wedge ^3 P^1_\g) }\bigoplus {\ker}\nb^4|_{ \big(1\ox (  \wedge ^4 P_\g^1) \big) }\,.
\end{align*}	
Now the  isomorphism follows from Lemma \ref{cokernel-part-in 4coh};
 \eqref{ker-nb4-middle-part},   \eqref{ker-nb4-middle-part-refine}, Lemma \ref{ker-nb4-complemntry-part}; and Lemma \ref{lem-ker-im-nb2}(1).

In $(1)$,  the proof follows from the isomorphism together with the fact that $\g=[\g,\g]+\h$.
Furthermore, $(2)$ follows   from   Lemma \ref{lem-decomposition-of-S(h)-new}(4). 
This completes the proof.
 \qed

\section{Computations and applications}\label{appl-thm}
In this section, we will see that Theorems \ref{thm-3rd-coh-general} and \ref{thm-4th-coh-general} transform the problem of computing the third and fourth cohomologies of a general compact homogeneous space into the problem of determining the rank of a specific associated matrix. Recall that the dimensions of the associated vector spaces, namely, $\NC_{\g,\h}$ and $\CC_{\g,H}$ as defined in Definition \ref{def-NC-gH}, play a crucial role in computing the third and fourth Betti numbers. In the following discussion, we will
demonstrate how these dimensions can be computed using the rank of a matrix naturally associated with the pair $(G, H)$. We will also provide certain general reduction processes which will facilitate the computation of cohomologies.

\subsection{Some general considerations for compact case computations}\label{computation-compact-case}
In this section, we establish some general principles that will be useful in constructing the matrix naturally associated with the pair $(G, H)$.

We need some notions associated to  representation of a compact Lie group over finite dimensional vector spaces over $\R$;  see pp. 93--101, \cite{BD} for details.

\begin{definition}\label{type-definition}
	Let $K$ be a compact Lie group not necessarily connected, and 
	$V$ be a finite dimensional vector space over $\R$
	on which $K$ is acting
	$\R$-linearly. Assume further $V$ is an irreducible representation of $K$. Then by Schur's lemma it follows that  ${\rm End}_K V$ is a division algebra over $\R$. The 
	irreducible representation $V$ of $K$ is called {\it real, complex or quaternionic type}
	according as the ${\rm End}_K V $ is isomorphic to $\R$,  $\C$ or $\H$  as $\R$-algebras, respectively. In particular $\dim {\rm End}_K V =1, 2 \text{ or } 4$ as a $\R$-vector space.      \qed
\end{definition}

In the next lemma we record certain facts related to the space of invariant bilinear forms  on a underlying vector space of a representation of a compact Lie group.

Let $V, W$ be  $\R$-representations of $K$. Let {\it $P (V, W)^K $ be the space of $K$-invariant $\R$-bilinear pairing}, that is, maps $\phi : V \times W \to \R$ such that $\phi$ is $\R$-linear in each variable and $\phi (g \cdot v, g \cdot w) = \phi (v,w)$ for all $g \in K, (v,w) \in V \times W$. If $V, W $ are irreducible $K$- representations, then it can be deduced that $\dim S^2 (W^*)^K =1$,
$\dim P (V, W)^K = \dim {\rm End}_K V$ provided  $V \simeq W$ as $K$-modules, and 
$\dim P (V, W)^K =0$, otherwise. 

Recall that, if  $V^K $ denotes the space of fixed points in $V$ under the action of $K$, that is, $V^K\,=\, \{ v \,\mid\, g\cdot v\, =\,v \text{ for all } g \,\in\, K\}$ and let 
$(K -1) V$ be the $\R$-span of the elements of the set $\{ g\cdot v -v \,\mid\, g \,\in\, K, \, v \,\in\, V\}$, then $(K-1)V$ is $K$-invariant and $V = V^K \oplus (K-1) V$; see Proposition \ref{representation-proposition}.

\begin{proposition}\label{real-cplx-quat} 
Let $K$ be a compact Lie group not necessarily connected and 	$V$ be a finite dimensional vector space over $\R$	on which $K$ is acting	$\R$-linearly.
\begin{enumerate}
\item 
Let $V$ be an isotypical $K$-module and $n$ be the number of irreducible 	components of $V$.
Let us fix an  irreducible $K$-subspace of $V$, say, $W$.
Let $V = W_1 \oplus \cdots \oplus W_n$ be a decomposition of $V$ into irreducible $K$-invariant subspaces $W_i \simeq W$. Then 
		$$
		S^2 (V^*)^K \simeq \bigoplus_i S^2 (W_i^*)^K \oplus \bigoplus_{i < j} P(W_i , W_j)^K
		$$
		In particular we have
		$$\dim S^2 (V^*)^K = n + \frac{n^2-n}{2} \, \dim {\rm End}_K W.
		$$
		More precisely, $\dim S^2 (V^*)^K = n(n+1)/2, \,  n^2$ or $ 2 n^2-n$ according as
		$W$ is real, complex or quaternionic type, respectively.
		\item 
		Let $V_1, \dots, V_r$ be an enumeration of all the distinct isotypical components
		of $(K-1)V$. 
		Then
		$$
		S^2 (V^*)^K \simeq  S^2 ((V^K)^*) \oplus \bigoplus_{i=1}^r S^2 (V_i^*)^K.
		$$
		\begin{enumerate}
			\item Let $W_i \subset V_i$ be an irreducible $K$-subspace and $n_i$ be the number of irreducible factors of $V_i$ for all $i$. 
			Let $q := \dim V^K$.
			Then
			$$\dim S^2 (V^*)^K = 
			\frac{q (q+1)}{2} + \sum_i  (n_i + \frac{n_i^2-n_i}{2} \, \dim {\rm End}_K W_i ). 
			$$
			\item Thus we have
			$$
			\frac{q (q+1)}{2} +  \sum_i \frac{n_i (n_i +1)}{2}   \,\leq  \, \dim S^2 (V^*)^K  \, \leq \,  \frac{q (q+1)}{2} + \sum_i 2 n_i^2- n_i.
			$$
		\end{enumerate}
	\end{enumerate}
\end{proposition}

\begin{proof}
	The proof of the first part of (1) is straightforward and hence omitted. The next part follows from the fact that if $W$ is
	irreducible then $\dim S^2 (W^*)^K =1$ and 
	the dimension of the space of all $K$-invariant bilinear forms on $W$ is $\dim {\rm End}_K W$.
	
	The proof of (2) follows from (1) and the fact that 
	for all $i, j$ with $i \neq j$ the spaces $V_i$ and $V_j$ are orthogonal with respect to any $K$-invariant bilinear form on $V$.
\end{proof}

\begin{definition}\label{definition-invariant-min-ideal}
	Let $K$ be a compact Lie group, not necessarily connected, with Lie algebra $\k$. Let $\a \subset [\k, \k]$ be an ideal of $\k$. We denote the set of simple factors of $\a$ by $\IC(\a)$ and the number of simple factors in $\a$ by $\#\a$. Furthermore, if $\a$ is invariant under the adjoint action of $K$, we define $\#(\a, K)$ as the number of nonzero minimal $K$-invariant ideals of $\a$.\qed
\end{definition}

Let $K$ be a compact Lie group, not necessarily connected, and let $\l$ be a Lie subalgebra of $\k$ containing the commutator $[\k, \k]$. Then $\z(\l) = \z(\k) \cap \l$ and $[\l, \l] = [\k, \k]$. We assume further that $\l$ is invariant under the adjoint action of $K$.  
Let $t = \#[\k, \k] = \#[\l, \l]$ and then  $\l = \z(\l) \oplus \k_1 \oplus \cdots \oplus \k_t$, where $\k_i$ (for $i = 1, \dots, t$) are the simple factors of $[\k, \k] = [\l, \l]$.

Consider the adjoint action of $K$ on the set $\IC([\k, \k])$ consisting of the simple factors of $[\k, \k]$. Since $K^0$ acts trivially on the finite set $\IC([\k, \k])$, the action of $K$ factors through the finite group $K/K^0$. Let $k$ be the number of finitely many
orbits under the action of $K/K^0$ on $\IC([\k, \k])$. We fix an enumeration of the
set of orbits, and let $\p_i$ denote 
the sum of all the simple ideals in the $i$-th orbit, for all $1 \leq i \leq k$.
In particular, for all $i = 1, \dots, k$, the simple ideals of $\p_i$ are all isomorphic. It follows that if $[\k, \k] \neq 0$, then $\p_1, \dots, \p_k$ are precisely the minimal $K$-invariant nonzero ideals of $[\k, \k]$. 
Further,  we have $\l = \z(\l) \oplus \p_1 \oplus \cdots \oplus \p_k$.

\begin{lemma}\label{lem-decomposition-of-S(h)-new}
Let $K$ be a compact Lie group, not necessarily connected. Let $\l \subset \k$ be a Lie subalgebra such that $[\k, \k] \subset \l$, and assume that $\l$ is invariant under the adjoint action of $K$. Let $t:= \#[\k, \k]$. Using the notation introduced above, 	we define $\widetilde{\Bi}^\l _{\k_i}$ for $1 \leq i \leq t$ 
	and $\widetilde{\Bi}^\l _{\p_j}$ for  $1 \leq j \leq k$ as in \eqref{definition-B}. Then the following statements hold:
	
	\begin{enumerate}
		\item  \label{ref-lem-1} $S^2 (\l^*)^{K^0} = S^2 (\z(\l)^*) \bigoplus S^2 ( \k_1^*)^{K^0} \oplus \dots \oplus S^2 ( \k_t^*)^{K^0}$.
		\begin{enumerate}
			\item \label{ref-lem1-a} For all $i = 1, \dots , t$, the space $S^2 ( \k_i^*)^{K^0}$ is the $\mathbb{R}$-span of $\widetilde{\Bi}^\l _{\k_i}$, implying \\  $\dim  S^2 ( \k_i^*)^{K^0} = 1$.
			\item 	\label{cor-basis-Bi}  Moreover, $\{ \,\widetilde{\Bi}^\l _{\k_i} \mid 1\leq i \leq \#[\k, \k] \,\}$ forms a basis of $S^2 ([\k, \k]^*)^{K^0}$.
		\end{enumerate}
		
		\item  \label{ref-lem-2} $S^2 (\l^*)^K = S^2 (\z(\l)^*)^{K/K^0} \bigoplus S^2 ( \p_1^*)^{K} \oplus \dots \oplus S^2 ( \p_k^*)^{K}$.
		\begin{enumerate}
			\item \label{ref-lem1-a-2} For all $i = 1, \dots , k$, the space $S^2 ( \p_i^*)^{K}$ is the $\mathbb{R}$-span of $ \widetilde{\Bi}^\l _{\p_i}$, implying  \\ 
			$\dim S^2 ( \p_i^*)^{K} = 1$.
			\item 	\label{cor-basis-Bi-2}  Moreover, $\{\, \widetilde{\Bi}^\l _{\p_i} \mid 1\leq i \leq k \,\}$ forms a basis of $S^2 ([\k, \k]^*)^{K}$. Thus   $\dim  S^2 ([\k, \k]^*)^K $ $= k$.
		\end{enumerate}
		
		\item  \label{ref-lem-3} $\dim S^2 (\l^*)^K = \dim S^2 (\z(\l)^*)^{K/K^0} + k$.
		
		\item  \label{ref-lem-4} If $K$ is either connected or all the simple factors of $[\k, \k]$ are pairwise non-isomorphic, then $\dim S^2 (\l^*)^K = \dim S^2 (\z(\l)^*)^{K/K^0} + \# [\k, \k]$. In particular, if $K^0$ is simple, then $\dim S^2 (\k^*)^K = 1$.
	\end{enumerate}
\end{lemma}

\begin{proof}
We first prove  \eqref{ref-lem-1}. Observe that if $B \,\in\, S^2 (\l^*)^{K^0}$ then $B$ is $\l$-invariant.  As $\k_i$ is a simple ideal  	for all $i$,  by Lemma \ref{lem-orthog-ideals},  it is 	immediate that $B ( \k_i, \k_j)=0=	B ( \z (\l), \k_k) $,  for all $i,j$ with $i \neq j$ and for all $k$.  Thus we have 	$S^2 (\l^*)^{K^0} = S^2 (\z(\l)^*) \bigoplus S^2 ( \k_1^*)^{K^0} \oplus \dots \oplus S^2 ( \k_t^*)^{K^0}$.

To see a proof of  \eqref{ref-lem1-a}, 	first observe that compactness of $K$  forces the commutator $[\k,\, \k]$ to be a compact Lie algebra. 	Thus the simple factor $\k_i$ is also compact, and this in turn implies $\k_i$ to be absolutely simple, that is,  $\C \ox_\R \k_i$ is a simple Lie algebra over $\C$. Thus by Schur's lemma 	any $K^0$-invariant symmetric bilinear form on $\k_i$ is a (real) scalar multiple of the Killing form  on $\k_i$. The proof now follows from the fact that the Killing form 	$\widetilde{\Bi}_{\k_i}^\l$ on $\k_i$ is $K^0$-invariant.  	The proof of \eqref{cor-basis-Bi} follows from \eqref{ref-lem-1} and \eqref{ref-lem1-a}.

We next prove \eqref{ref-lem-2}.  To prove the first part of \eqref{ref-lem-2} we use Lemma \ref{lem-orthog-ideals}  and the proof follows analogously as the proof of the first part of \eqref{ref-lem-1}.

We now prove the part  \eqref{ref-lem1-a-2}. For notational simplicity we will assume	$i=1$, and show that $S^2(\p_1^*)^K$ is spanned by $ \widetilde{\Bi}_{\p_1}^\l$.	Let $\p_1=\k_{i_1}\oplus \dots\oplus \k_{i_{n_1}}$, and $S_1:=\{ \k_{i_1}, \dots, \k_{i_{n_1}}\}$. Then $K$ acts transitively on $S_1$. Any arbitrary element in $S^2(\p_i^*)^K$ is of the following form:
	$$
	c_{i_1}\widetilde {\Bi}_{\k_{i_1}}^\l  + \dots +  c_{i_{n_1}}\widetilde {\Bi}_{\k_{i_{n_1}}}^\l \,,
	$$
where $c_{i_j}\in \R$ for $1\leq j\leq n_1$.  Since $K$ acts transitively on $S_1$, 	it follows that all the coefficients are equal, that is, $ c_{i_1} \,=\, c_{i_j}$, for  $1\leq j\leq n_1$. Thus the above sum turns out to be $c_{i_1}\widetilde {\Bi}_{\p_1}^\l$, and in particular,   
	$$
	S^2(\p_1^*)^K\,=\,  \R\sum_{  1\leq j\leq n_1} \widetilde {\Bi}_{\k_{i_j}}^\k   \,=\, \R \widetilde {\Bi}_{\p_1}^\k \,.
	$$
	The proof of the part \eqref{cor-basis-Bi-2} follows from \eqref{ref-lem-2} and \eqref{ref-lem1-a-2}. 	The proof of  \eqref{ref-lem-3} is immediate from \eqref{ref-lem-2}. The  proof of \eqref{ref-lem-4} follows  from \eqref{ref-lem-3}, since  $K$ acts trivially on  $\IC ([\k,\k])$ in this case.   
\end{proof}

\subsection{The first reduction}\label{first-reduction}
In this section, we describe a basic reduction process which helps facilitating the computation of
cohomologies of compact connected homogeneous spaces. This reduction allows us to 
express any such spaces
as a product of a compact Lie group, and another homogeneous space, up to a finite covering.

We formulate some basic definitions for clarity.

\begin{definition}\label{smallest connected normal subgroup}
	Let $G$ be a  connected Lie group, and let $C$ be a subgroup of $G$, not necessarily closed or connected. Then
	it is clear that there exists a minimal dimensional closed, connected, normal subgroup of $G$ containing $C$. This subgroup is unique if $C$ is connected, and not necessarily so, otherwise.
	We call this subgroup a \emph{connected normal closure} of $C$ in $G$.
	
	If $A$ is a closed connected normal subgroup of $G$, then a closed connected normal subgroup $B$ is called a \emph{connected normal complement} of $A$ if $G = A \cdot B$ and $A \cap B$ is a discrete (hence finite central) subgroup of $G$.
	If $G$ is compact connected and $A$ is as above then it is easy to see that a
	connected normal complement of $A$ 
	always exists, and, 
	notably, if $A$ contains the connected center $Z(G)^0$ of $G$, then $A$ admits a unique connected normal complement.   \qed
\end{definition}

\begin{remark}\label{remark-first-reduction}
	In Definition \ref{smallest connected normal subgroup}, if $H$ is further assumed to be connected, then the Lie algebra of its connected normal closure is given by
	\begin{equation}
		\pi^\g_{\z(\g)} (\h) + \sum_{ i\in J_\h} \g_i,
	\end{equation}
	where $J_\h := \{ i \mid \pi^\g_{\g_i} (\h) \neq 0 \}$ (see  \eqref{notation-projection} for notation).    \qed
\end{remark}

Let $G$ be a compact connected Lie group, and let $H$ be a closed subgroup, not necessarily connected. 
For the compact homogeneous space $G/H$, one can obtain a product structure (up to a finite covering)   as follows.  Let $F$ be a connected normal closure of $H$, and let $E$ be a connected normal complement of $F$. The finite central subgroup $E \cap F$ of $G$ acts on $E \times (F/H)$ by
\begin{equation}\label{action}
	\gamma \cdot (x, y H) := ( \gamma x, \gamma^{-1} yH),
\end{equation}
for all $\gamma \in E \cap F$ and $(x, y H) \in E \times (F/H)$. It can be verified that this action  is free, and the map
\begin{align}\label{first reduction covering map}
	\delta : E \times (F/H) \to G/H, \quad \text{defined by} \quad \delta (x, y H) := xy H,
\end{align}
is a covering map, where the fibers are exactly the orbits of the action of 
$E \cap F$ on $E \times (F/H)$ as in \eqref{action}.

The following lemma when used with the K\"{u}nneth's theorem for de Rham cohomology reduces
computing the first four Betti numbers of $G/H$  to the case where $G$ is the connected normal closure of $H$.

\begin{lemma}\label{first-reduction-lemma}
	Let $G$ be a compact connected Lie group, and let $H$ be a closed subgroup, not necessarily connected. 
	Let $F$ be a connected normal closure of
	$H$ and $E$ be a connected normal complement
	of $F$ in $G$.
	Then,   $i = 1, 2,  3, 4$,
	\begin{equation}
		H^i (G/H)=	H^i (E \times (F/H)).
	\end{equation}
\end{lemma}

\begin{proof}
	Let $H_1 := \{e\} \times H$ and $\Gamma := 
	\{ (\gamma, \gamma^{-1}) \mid \gamma \in E \cap F \} $. Then $\Gamma H_1 \subset E \times F$ is a closed subgroup, and moreover, 
	$G/H$ and $(E \times F ) / \Gamma H_1$
	are diffeomorphic. Now the proof follows directly from Theorems \ref{thm-3rd-coh-general} and \ref{thm-4th-coh-general}, along with \cite[Theorems 3.3, 3.6]{BCM} and the fact that 
	$E \cap F \subset Z(G)$.
\end{proof}

We next record a consequence of the above lemma.

\begin{theorem}\label{first-reduction-theorem}
	Let $G, H, F, E$ be as in the Lemma \ref{first-reduction-lemma}. Further assume that $\z(\g) \subset 
	\h$ (in particular we may also assume $G$ to be semisimple).  
	Let  $F'$ be the connected normal closure of
	$H^0$ and $E'$ be the connected normal complement
	of $F'$ in $G$ (the condition  $\z(\g) \subset 
	\h$ forces $E, E'$ to be semisimple).
	\begin{enumerate}
		\item Then  $H^3 (G/H) \simeq  H^3 (G/H^0) \simeq	H^3 (E \times (F/H)) \simeq  H^3 (E' \times (F'/H^0))$ and moreover,
		$$\dim H^3 (G/H) = \# \e + \dim H^3 (F/H) = \#\e + \dim H^3 (F/H^0).$$	
		\item We also have $\dim H^4 (G/H) = \dim H^4 (F/H).$		
	\end{enumerate}		
\end{theorem}
\begin{proof}
The proofs follow from the facts that  	$E, E'$ are semisimple, 	Lemma \ref{first-reduction-lemma} and straightforward computations using 	Theorems \ref{thm-3rd-coh-general} and	\ref{thm-4th-coh-general}.
\end{proof}

\subsection{The matrix associated to the pair $(\g, H)$ and a basis of 
	$S^2( (\z(\h) \cap [\g, \g])^*)^H$}
\label{associated matrix}

Let $G$ be a compact connected Lie group, and $H \subseteq G$ be a closed subgroup, which is not necessarily connected.  
Let us now return to our objective of calculating  $ \dim \NC_{\g,\h} $ and $\dim \CC_{\g,H}$. We will apply Lemma \ref{lem-decomposition-of-S(h)-new}, and the associated setting for $K := H$ and $\l := \h \cap [\g, \g]$.

Let $\Psi_{\g, \h} : S^2([ \g, \g]^*)^G \to S^2( (\h \cap [\g, \g])^*)^H$ be the restriction map as defined in \eqref{map-Psi-g-h}. Using Lemma \ref{lem-decomposition-of-S(h)-new}, since 
$$
S^2( (\h \cap [\g, \g])^*)^H = S^2( (\z(\h) \cap [\g, \g])^*)^H \oplus S^2( [\h, \h]^*)^H,
$$ 
we may decompose $\Psi_{\g, \h} = \Psi_{\g, \h}^1 \oplus \Psi_{\g, \h}^2$, where:

\begin{itemize}
	\item $\Psi_{\g, \h}^1 : S^2([ \g, \g]^*)^G \to S^2( (\z(\h) \cap [\g, \g])^*)^H$ is given by
	\[
	\Psi_{\g, \h}^1 (\phi) := \phi|_{(\z(\h) \cap [\g, \g]) \times (\z(\h) \cap [\g, \g])},
	\]
	
	\item $\Psi_{\g, \h}^2 : S^2([ \g, \g]^*)^G \to S^2([ \h, \h]^*)^H$ is given by
	\[
	\Psi_{\g, \h}^2 (\phi) := \phi|_{[\h, \h] \times [\h, \h]}.
	\]

\item Let $m := \# [\g, \g]$, and decompose $[\g, \g] = \g_1 \oplus \cdots \oplus \g_m$  
into the direct sum of simple ideals. 
Let $k$ be the number of orbits under the adjoint action of $H$ on the finite set $\IC([\h, \h])$ (as the action of $H^0$ is trivial, this action factors through the finite group $H/H^0$), and let $\p_1, \dots, \p_k$ denote the sum of the simple ideals in each of the $k$ orbits.
\end{itemize}

Consider the bases $\BC_\g := \{ \tilde{\Bi}_{\g_j}^{\g} \mid 1 \leq j \leq m \}$ for $S^2([ \g, \g]^*)^G$ and $\BC_H := \{ \tilde{\Bi}_{\p_i}^{\l} \mid 1 \leq i \leq k \}$ for $S^2([ \h, \h]^*)^H$ (see Lemma \ref{lem-decomposition-of-S(h)-new}). Note that the basis for $S^2([ \h, \h]^*)^H$ depends on how the finite group $H/H^0$ acts on the finite set $\IC([\h, \h])$.

\begin{itemize}
	\item The action of $H$ on $\z(\h) \cap [\g, \g]$ factors through $H/H^0$. 	
	Let $\BC$ be a basis for $S^2((\z(\h) \cap [\g, \g])^*)^H$ (this can be done using Lemma \ref{real-cplx-quat}).  
	Let  $q := \dim(\z(\h) \cap [\g, \g])^H$,  let $r$
	be the number of distinct isotypical component of the $\R$-representation $z(\h) \cap [\g, \g]$ of the
	of the finite group $H/H^0$ (which is obtained
	through the adjoint representation of $H$).  
	Let $n_i$ be the number of irreducible factor of $i$-th isotypical component
	of  $(H-1) (z(\h) \cap [\g, \g])$ for all $i = 1, \dots, r$; see Proposition \ref{representation-proposition} for the notation.  Then from Proposition
	\ref{real-cplx-quat} 
	it follows that
	\[
 \frac{q (q+1)}{2} + \sum_{i=1}^r 2 n_i^2-n_i \geq  
	 \dim S^2((\z(\h) \cap [\g, \g])^*)^H \geq \frac{q(q+1)}{2} + \sum_{i=1}^r  \frac{n_i (n_i+1)}{2}  \
	\]
	
	\item Let $p := \dim S^2((\z(\h) \cap [\g, \g])^*)^H$.  Let $A_1 := (a_{ij}) \in {\rm M}_{p \times m} (\mathbb{R})$ be the matrix of the linear map $\Psi_{\g, \h}^1$ with respect to the bases $\BC_\g$ and $\BC$.
	
	\item We define the $k \times m$ matrix $A_2 := (\alpha_{ij}) \in {\rm M}_{k \times m} (\mathbb{R})$ by
	\[
	{\widetilde{\Bi}_{\g_j}^\g}|_{[\h, \h] \times [\h, \h]} = \sum_{i=1}^k \alpha_{ij} \widetilde{\Bi}_{\p_i}^{\l}.
	\]
	Since the Killing form on a compact semisimple Lie algebra is negative definite, we conclude that $\alpha_{ij} > 0$ if and only if $\pi^\g_{\g_j} (\p_i) \neq 0$, and $\alpha_{ij} = 0$ otherwise. If $G$ is semisimple and $H$ is connected, then $k = \#[\h, \h]$, and in this case, the above constants also appear in \cite[p. 175]{LW} in a slightly different format. Thus, $A_2$ is the matrix for the linear map $\Psi_{\g, \h}^2$ with respect to the bases $\BC_\g$ and $\BC_H$.
\end{itemize}

We are now in a position to define the matrix $A_{\g, H}$.

\begin{definition}\label{definition-of-associated matrix}
	{\rm Let $A_{\g, H}$ be the matrix in ${\rm M}_{(p+k) \times m} (\mathbb{R})$ such that the upper $p \times m$ block is the matrix $A_1$ and the lower $k \times m$ block is the matrix $A_2$.  Thus,
		$$ A_{\g, H} :=
		\begin{bmatrix}
			A_1  \\
			A_2 \\
		\end{bmatrix}. 
		$$ 
		Then $A_{\g, H}$ is the matrix of the map $\Psi_{\g, \h}$ with respect to the bases $\BC_\g$ of $S^2([ \g, \g]^*)^G$ and $\BC \cupdot \BC_H$ of $S^2((\h \cap [\g, \g])^*)^H$.}     \qed
\end{definition}

\begin{remark}{\rm We now make two remarks on the above matrix.
\begin{enumerate}
\item
Only ${\rm Rank} \, A_{\g, H}$ and the dimension of $S^2( (\h \cap [\g, \g])^*)^H$ are relevant when determining these dimensions. When $G$ is  semisimple ${\rm Rank} \,
A_{\g, H}$   only ingredient need to compute   $\dim H^3 (G/H)$.
Regarding the computation of the third and fourth cohomologies of $G/H$ for a general compact connected group $G$ modulo a general closed subgroup $H$, the above represents the optimal level of understanding attainable. In specific settings, sharper results may be obtained (see Remark \ref{remark-LW} below).

\item
Observe that the bases $\BC_\g$ of $S^2([ \g, \g]^*)^G$, $\BC_H$ of $S^2([ \h, \h]^*)^H$, and, consequently, the matrix $A_2$ are naturally chosen using the Killing form of the simple factors of the compact Lie algebras. However, the basis $\BC$ of $S^2((\z(\h) \cap [\g, \g])^*)^{H/H^0}$, and consequently the matrix $A_1$, is not a natural one and depends on how $\h$ is positioned in $\g$.
Moreover,  the cardinality of the basis $\BC$
depends on the types (see Definition \ref{type-definition}) and number of irreducible factors
in each isotypical component of the 
$\R$-representation of the finite
group $H/H^0$ on the $\R$- vector space 
$\z(\h) \cap [\g, \g]$; see also Lemma \ref{real-cplx-quat}
in this context. 
Thus, if $\z(\h) \cap [\g, \g] \neq 0$, the matrix $A_{\g, H}$ depends on the choice of a basis for $S^2((\z(\h) \cap [\g, \g])^*)^H$.    \qed
\end{enumerate}
}  
\end{remark}

In light of the above considerations, we now record some useful consequences for the lower bounds of $\dim \NC_{\g,\h}$. 

\begin{proposition}\label{rank of the restriction map}
Let $G$ be a compact connected Lie group and $H$ be a closed subgroup, not necessarily connected.	Let 	$q := \dim(\z(\h) \cap [\g, \g])^{H/H^0}$, and
 $p := \dim S^2((\z(\h) \cap [\g, \g])^*)^H$. Let $m := \# [\g, \g]$, and $k$ be the number of orbits under the adjoint action of $H/H^0$ on the finite set $\IC([\h, \h])$.
Furthermore, consider the representation 	of the finite group $H/H^0$ on the $\R$-vector space $z(\h) \cap [\g, \g]$, and let 
$r$ 	be the number of distinct isotypical component of the $H$-invariant subspace $(H-1)(z(\h) \cap [\g, \g])$; see Proposition \ref{representation-proposition} for the notation used.   We fix an enumeration of the isotypical components of   $(H-1)(z(\h) \cap [\g, \g])$, and 	let $n_i$ be the number of irreducible factors  of the $i$-th isotypical component	of  $(H-1)(z(\h) \cap [\g, \g])$ for all $i = 1, \dots, r$. 	
\begin{enumerate}
		\item 	Then
		\[
		\dim \NC_{\g,\h} \geq m - (k + p) \geq m - \left( k +  \frac{q (q+1)}{2} +
		 \sum_{i=1}^r 2 n_i^2- n_i \right).
		\]

		\item In particular, if $H/H^0$ acts transitively on $\IC([\h, \h])$, then $k = 1$ and hence 
		\[
		\dim \NC_{\g,\h} \geq m - (1 + p).
		\]
	\end{enumerate}
\end{proposition}

The following corollary highlights an interesting consequence of the preceding proposition.

\begin{corollary}\label{cor-finitegp1}
Let $G$ be compact semisimple, and $H$ be a closed subgroup which is not necessarily connected. Let $m := \#\g$ and $p := \dim S^2(\z(\h)^*)$.   If $k$ is the number of orbits
		under the adjoint action of $H/H^0$ on the finite set $\IC([\h, \h])$.  Then \[
		\dim H^3(G/H) = \dim \NC_{\g,\h} \geq m - (k + p).
		\]
\end{corollary}

\begin{remark}\label{remark-LW} In view of notions introduced in this section
	we now make some remarks on certain results from \cite{LW}. In \cite{LW}, it is assumed that $G$ is semisimple and that $H$ is connected. Throughout this remark, we shall adopt these assumptions.
	
	\begin{enumerate}
		\item In \cite[Definition 4.7]{LW}, the notion of an ``aligned" homogeneous space is introduced via three conditions. However, only conditions (i) and (ii) are essential, as their fulfilment is equivalent to satisfying all three conditions (i), (ii), and (iii). This equivalence can be readily established by appropriately scaling the $s$-tuple $(c_1, \dots, c_s)$ appearing in \cite[Definition 4.7]{LW} by a positive real number, if necessary.
		
		\item In view of  Theorem \ref{first-reduction-theorem},
		in computing the third (and fourth) Betti number  of $G/H$ 
		only the connected normal closure of $H$ in $G$ matters; see Definition \ref{smallest connected normal subgroup}. 
		If $E$ is the connected normal closure of $H$ in $G$ then using Theorem  \ref{first-reduction-theorem} and Theorem \ref{thm-3rd-coh-general}
		it is easy to see that the third Betti number of $G/H$ is
		$\# \g -1$ if and only if the third Betti number of $E/H$ is
		$\# \e -1$, which is equivalent to ${\rm Rank } 
		A_{\g, H} = 1= {\rm Rank } 
		A_{\e, H}$.  Furthermore,  from Remark \ref{remark-first-reduction}
		it is also immediate that, in this case, $E/H$ is ``aligned" in the
		in the sense of \cite[Definition 4.7]{LW}.  This is 
		exactly the equivalence of (i) and (iii) of \cite[Proposition 4.10]{LW}.  Thus ``alignment" of
		$G/H$ is the same as  ${\rm Rank } 
		A_{\g, H} = 1$ and that $G$ is the connected normal closure of $H$.
		
\item An ambiguity arises in \cite[Proposition 4.10]{LW} concerning the implication $\text{(ii)} \Rightarrow \text{(i)}$. Specifically, if $G$ and $K$ satisfy condition (ii), it is not explicitly stated what the authors mean by $G'$ and $H$.

 However, based on the proof of \cite[Proposition 4.10]{LW}, it appears that $G'$ denotes the  connected normal closure of $K$, while $H$ represents the connected normal complement of $G'$; see Definition \ref{smallest connected normal subgroup}.
		Under this interpretation, the implication $\text{(ii)} \Rightarrow \text{(i)}$ does not hold unconditionally but rather up to a finite covering (see the following remark). A counterexample is provided as follows.
		
		Let $A$ be a compact, simple, simply connected Lie group with a nontrivial center $Z(A)$; for example, one may assume, $A := \text{SU}_2$. Let $\Gamma := \{(x,x) \mid x \in Z(A)\}$, $G := (A \times A) / \Gamma$, and 
$p : A \times A \to G$ be the associated quotient homomorphism. Consider the embeddings  defined by :
		\[
		i_1: A \to G, \quad a \mapsto (a,e)\Gamma, \quad \text{and} \quad i_2: A \to G, \quad a \mapsto (e,a)\Gamma.
		\]
		Let $T$ be a maximal torus of $A$. We set $K := i_1(T)$. Since $Z(A) \subset T$, it follows that $p^{-1} (i_1(T)) = T \times Z(A)$. Thus
		\[\quad \frac{G}{K} = \frac{G}{i_1(A)} \simeq \frac{A \times A}{ p^{-1} (i_1(T)) } \simeq \frac{A \times A}{T \times Z(A)} \simeq (A/T) \times (A/Z(A)).
		\]
		Moreover, we have $\dim H^3(G/K) = 1$. It is easy to see that
the smallest closed connected normal subgroup of $G$ containing $K$ is $i_1(A)$, and the connected closed normal complement of $i_1(A)$ is $i_2(A)$. Now we have
the following observations on the fundamental groups :  $\pi_1(i_1(A)/K) = 1$, $\pi_1(i_2(A)) = \pi_1(A) = 1$, but $\pi_1(G/K) = \pi_1(A/T) \times \pi_1(A/Z(A)) \simeq Z(A) \neq 1$.
Thus $G/K$ is not homotopy equivalent to $(i_1(A)/K) \times i_2(A)$, and in particular, the spaces are not homeomorphic.
		
		\item In view of the ambiguity pointed out above,
		we next indicate  a way to modify (i) of \cite[Proposition 4.10]{LW}.  We follow the notation of 
		\cite[Proposition 4.10]{LW}.  Let $G'$
		be the connected normal closure of $K$ in $G$, and that $H$ is its connected normal complement; see Definition \ref{smallest connected normal subgroup}.
		 Then 
		using (2) above, 
		\ref{first reduction covering map},  one concludes that $b_3(G/K) = \#\g - 1$  if and only if 
		$b_3(G'/K) = \#\g' - 1$.  This, in turn, is equivalent to asserting  that
		there exists a covering map from $H \times (G'/K)$ to $G/K$, where the fibers correspond to orbits of the finite central subgroup $G' \cap H$ acting on $H \times (G'/K)$ (see \ref{first reduction covering map}). This map is not a homeomorphism unless $G' \cap H$ is trivial. The ``alignment" of $G'/K$ 	is obvious.    \qed
	\end{enumerate}
\end{remark}

\subsection{Computations on some special settings}\label{computation in special setting}
Now we apply our main results, namely, Theorem \ref{thm-3rd-coh-general}, and Theorem \ref{thm-4th-coh-general} in various
special cases of homogeneous spaces which occur naturally and record some interesting consequences.

\begin{corollary}\label{cor-h<gg}
	Let $G$ be a compact connected Lie group, and $H$  a closed subgroup of  $G$.
	\begin{enumerate}
		\item \label{cor-h<gg1}   Assume that $H$ is a subgroup of  $ (G,\,G) $, the commutator subgroup of $ G $. Then 
		\begin{align*}
			H^3(G/H )\, \simeq & \big({\z}(\h) ^*\big)^{H/H^0} \ox \z(\g)^*     \bigoplus \NC_{\g,\h}\,\bigoplus  \, \wedge^3 {\z(\g)}^* \,, \vspace*{0.05cm} \\  
			H^4(G/H )\,\, \simeq & \,\,  \big({\z}(\h)^*\big)^{H/H^0} \ox\wedge^2 \z(\g)^* \,\bigoplus   \, \,(\NC_{\g,\h})^{\dim  \z(\g)} \,\bigoplus \, \CC_{\g,H} \,\bigoplus  \,	\wedge^4  \z(\g)^* \, .
		\end{align*}
		\vspace*{0.01cm}
		\item Recall that $ Z(H) $ denotes the center of the  Lie group $ H $.  Let $r_0\,:=\,\dim   \big(\frac{\g}{[\g,\,\g]+\h}\big)  $.  Assume that  $ Z(H)\cap (G,G) $ is a finite group. 
		Then we have the following isomorphisms:
		\begin{align*}
			H^3(G/H )\, \simeq &   \, \,\NC_{\g,\h}\,\bigoplus  \,    \wedge^3{\Big(\frac{\g}{[\g,\,\g]+\h}\Big)}^*       \,, \\
			H^4(G/H )\,\, \simeq &  \,(\NC_{\g,\h})^{r_0}	 \, \bigoplus \, \CC_{\g,H}
			\,\bigoplus  \,	\wedge^4{\Big(\frac{\g}{[\g,\,\g]+\h}\Big)}^*  .
		\end{align*}
		\vspace*{0.01cm}
		\item\label{cor-dim-zg<1}
		Assume that $\dim   \z(\g) \,\leq\, 1$.
		Then we have the following isomorphisms: 
		\begin{align*}
			H^3(G/H )\,\,& \simeq\,\,   \big(({\z}(\h) \cap [\g,\,\g])^*\big)^{H/H^0} \ox {\Big(\frac{\g}{[\g,\,\g]+\h}\Big)}^* 
			\bigoplus \NC_{\g,\h}\,,\\
			H^4(G/H )\,\,& \simeq\,\,   
			\begin{cases}                          
				\CC_{\g,H}   &  {\rm if }\,\, {[\g,\,\g]+\h}=\g  \\
				\NC_{\g,\h} \,  \bigoplus\, \CC_{\g,H}  &  {\rm if }\,\, {[\g,\,\g]+\h}\,\subsetneqq\, \g\,.
			\end{cases}
		\end{align*}
	\end{enumerate}
\end{corollary}
\begin{remark}
	In this remark, we will indicate a natural situation where
	Corollary \ref{cor-h<gg}(\ref{cor-dim-zg<1}) can be applied.
	Let $L$ be a connected simple Lie group which is not necessarily compact,  and $M$  a closed
	subgroup with finitely many connected components.  Let 
	$G$ be a maximal compact subgroup of $L$ such that $H:= G \cap M$ is a maximal compact subgroup of $M$.  Then  $H^i (L/M, \R) \simeq H^i ( G/H,  \R)$ for all $i$; see \eqref{homotopy}. 
	It is well-known that  maximal compact subgroups of a simple real Lie group need not be semisimple, but the center of a maximal compact subgroup is of dimension at the most one; see \cite[Proposition 6.2, p. 382]{He}.  Thus 
	$\dim   \z(\g) \leq 1$,  and hence to calculate $H^i (L/M, \R),  \, i = 3,4$ we may invoke Corollary \ref{cor-h<gg}(\ref{cor-dim-zg<1}). \qed
\end{remark}

\begin{proposition}\label{prop-dimNgh}
Let $ G $ be a compact, connected Lie group, $ H\subset G$ a closed subgroup.  Let $r \,:=\, \#[\g,\, \g]$, and 
let $ [\g,\,\g]\,=\, \g_1\oplus \cdots \oplus \g_r $, where $ \g_i $'s are simple ideals in $ [\g,\,\g] $.
\begin{enumerate}
\item If $ \h \cap [\g,\,\g]  \,\neq\, 0 $,
then  $ \dim  \NC_{\g, \h}\, \leq \,  \#[\g,\,\g]-1$.
Furthermore,  if $H^0$ is simple then 
$ \dim  \NC_{\g, \h}\, = \,  \#[\g,\,\g]-1$,  and $ \dim  \CC_{\g,H}\,=\, 0 $.
\vspace*{0.2cm}

\item If $ k\,:=\, \#\{i \,\mid \,  \h\cap \g_i  \,\neq\, 0 \} $, then $ \dim  \NC_{\g, \h}  \,\leq \, \#[\g,\, \g]-k$.

\vspace*{0.2cm}
		
\item\label{prop-4.4-4} If $ \h \cap [\g,\,\g] \, \subset \,  \g_1\oplus \cdots \oplus \g_s $ for some $ s\,\leq\, \#[\g,\, \g] $, then 
$\dim  \NC_{\g, \h}  \,\geq \, \#[\g,\, \g]-s$.

\vspace*{0.2cm}

\item\label{prop-4.4-3} 

If $ \h \cap [\g,\,\g] \, \subset \,  \g_1\oplus \cdots \oplus \g_s$, and   $ \h\cap \g_i \,\neq\, 0$ for all $ i\,=\,1,\,\ldots ,\,s $, then
\begin{enumerate}
\item   $ \dim  \NC_{\g, \h}  \,= \, \#[\g, \,\g]-s$. 
 
\item  $ \dim  \Im \Psi_{\g,\h} \,=\, s$. In particular, $ s \,\leq\, \dim S^2((\h\cap [\g,\,\g]) ^*)^H$.  
\end{enumerate}	
\end{enumerate}
\end{proposition}

\begin{proof}
{\it Proof of (1).~} 
The proof follows from the fact that the map $ \Psi_{\g, \h} $ as in \eqref{map-Psi-g-h} is  a non-zero map and the Killing form
on $[\g,\,\g] $ is negative-definite. 

To prove the second part of (1) first observe that,  as $\h$ is simple and $ \h \cap [\g,\,\g]\,\neq\, 0 $,\,  $[\g, \,\g] \,\neq\, 0$.
Further, as $H^0$ is simple, by Lemma \ref{lem-decomposition-of-S(h)-new} (4),  $\dim S^2 (\h^*)^H \,=\,1$.  Let $B$ be the Killing from
of $[\g, \,\g]$.   Let $0 \neq X \,\in\, \h $.
As $B$ is negative definite, $B (X,\, X) \,\neq\, 0$.   This forces $\Psi_{\g, \h}$ to be non-zero and hence a surjective map. Thus 
$ \dim  \NC_{\g, \h}\, = \,  \#[\g,\,\g]-1$,  and $ \dim  \CC_{\g,H}\,=\, 0 $.

{\it Proof of (2).~} 
Recall that $ \{ \widetilde{\Bi} _i \,\mid\, 1\,\leq \,i \,\leq\, \#[\g,\,\g] \} $ form a basis of $ S^2 ([\g,\, \g]^*)^G$ where $ \widetilde{\Bi}_i $ is as in \eqref{definition-B-tilde}; see Lemma \ref{lem-decomposition-of-S(h)-new}\eqref{cor-basis-Bi}.   
Let $\g_{i_l},\,  l\,=\, 1,\,  \dots,\, k$ be distinct ideals of $[\g,\,  \g]$ such that $\h\cap \g_{i_l}  \,\neq\, 0$.
We choose  $0\,\neq\, X_{i_l} \,\in\, \h\cap \g_{i_l} $ for $ 1\,\leq\, l \,\leq\, k$. Since $ {\Bi}_{i_l} $ is negative-definite
on $ \g_{i_l} $ we have that $\widetilde{{\Bi}} _{i_l}(X_{i_l},   X_{i_l})  \,
=\,{\Bi}_{i_l}(X_{i_l}, \,  X_{i_l}) \,\neq \, 0$. Thus $ \dim  {\rm ker }\Psi_{\g, \h} \,\leq\, r-k$. 

{\it Proof of (3).}
Since  $ \h \cap [\g,\,\g] \, \subset\,   \g_1\oplus \cdots \oplus \g_s  $ for $s\leq r $,   we have 
$\widetilde{{\Bi}} _{j}|_{(\h \cap [\g,\,\g]) \times (\h \cap [\g,\,\g])}\, =\,0$ for $ j\,=\,s+1,\, \ldots , \,r$.
Now the first conclusion follows from Lemma \ref{lem-decomposition-of-S(h)-new}\eqref{cor-basis-Bi}. 

{\it Proof of (4).}
 The statements in (4) follows from (2), (3) and Definition \ref{def-NC-gH}.
\end{proof}

\begin{example}
The $n$-dimensional sphere $ \mathbb{S}^n $ is diffeomorphic to the homogeneous space $ {\rm SO}(n+1) /{\rm SO}(n)$. Recall that $ {\rm SO}(3) $, $ {\rm SO}(n) $ for $ n\geq 5 $ are simple, where as $ {\rm SO}(4) $ is not simple ($ \s\o(4) \simeq \s\u(2) \oplus \s\u(2) $). Thus using Corollary \ref{cor-h<gg}\eqref{cor-h<gg1} and Proposition \ref{prop-dimNgh}, it follows that for $ i=3,4 $, $ \dim H^i(\mathbb{S}^n)= 1$ if $ i=n $, and  $ \dim H^i(\mathbb{S}^n)= 0$ if $ i\neq n $.
\qed
\end{example}

{\bf Proof of Corollary \ref{cor-g-semisimple}.}
The proofs of \eqref{cor1.3-1}   follows from  Theorem \ref{thm-3rd-coh-general}, and Proposition \ref{prop-dimNgh}(1). 

The first part of \eqref{cor1.3-2}   follows from   Theorem \ref{thm-4th-coh-general}. Second part follows using Lemma \ref{lem-decomposition-of-S(h)-new}(4) and the fact that $ \Psi_{\g,\h} $  as in \eqref{map-Psi-g-h} is a non-zero map. \eqref{cor1.3-2a} follows from above as $ \h $ is semisimple. 
For \eqref{cor1.3-2b}, use the fact that $\dim {\Im} \Psi_{\g,\h} =1$, and $ \dim S^2(\h^*)^H=\#(\h,H) $. For \eqref{cor1.3-2c}, use Proposition \ref{prop-dimNgh}(1). \qed

{\bf Proof of Theorem \ref{thm-invariant}.}
The proofs of \eqref{cor1.3-3} and  \eqref{cor1.3-5} follow  from Theorem \ref{thm-3rd-coh-general}, Theorem \ref{thm-4th-coh-general}  and Lemma \ref{lem-decomposition-of-S(h)-new}.  The proof of \eqref{cor1.3-4} follows from \eqref{cor1.3-3} and Lemma \ref{lem-decomposition-of-S(h)-new}\eqref{ref-lem-4}.
\qed

{\bf Proof of Corollary \ref{cor-homotopy-inv}.}	
Proof of (1) follows from Theorem \ref{thm-invariant} \eqref{cor1.3-4}.  

We now prove (2).
Let $n : = \dim T $.  Since $H$ is a semisimple subgroup of $G$ it follows that $\#\h	\leq {\rm rank}\, H \leq {\rm rank}\, G
\leq n$.  
We will divide the proof in two parts. Fist we assume that $n>1$. 
In this case  using Theorem  \ref{thm-invariant} \eqref{cor1.3-4}, \eqref{cor1.3-5} and the fact that 
$\dfrac{n(n+1)}{2} >n$ for $n>1$, it follows that $\dim  H^4(G/H) -  \dim  H^3(G/H) \,\neq \, \dim  H^4(G/T) - \dim  H^3(G/T)$.  Thus $G/H$ and $G/T$ can not be homotopically equivalent.

We next consider the case when rank of $G$ is $1$.  The condition on the rank forces $G$ to be a simple Lie group, and hence  $\dim  H^2(G/T )=1 $; see \cite[Theorem 3.3]{BCM}. On the other hand  rank of $H$ has to be $1$, and hence $H$ is simple. Thus $\dim  H^2(G/H )=0 $. Thus, again,  $G/H$ and $G/T$ can not be homotopically equivalent. This completes the proof.
\qed

Now we include a particular  case where  one can say the dimension of the  third cohomology of $G/H$.
\begin{corollary}\label{cor-g-se}
	Let $G$ be a compact semisimple Lie group and $H$ be a closed subgroup of $G$  which is not necessarily connected and $\dim H\,>\,0$.  
	Moreover assume that  $\g_1,\, \cdots,\, \g_s$ are distinct simple ideals of $\g$ such that	$\h \,\subset\, \g_1 \oplus \cdots \oplus \g_s$ and that $ \h \cap \g_i \,\neq\, 0$ for all $i \,=\, 1,\, \cdots,\, s$.   Then	 $\dim  H^3(G/H  )	\,=\, \# \g - s$.  \vspace{0.1cm}		
\end{corollary}			

\begin{proof}
	The proof follows from Theorem \ref{thm-3rd-coh-general}, and Proposition \ref{prop-dimNgh} \eqref{prop-4.4-3}.
\end{proof}

In the next result, we specialize to the situation where $H^0$ is assumed to be semisimple.
\begin{corollary}\label{cor-h-gg34}
Let $ G $ be a compact, connected Lie group $ H\subset G$ be a closed  subgroup  such that $H^0$ is a semisimple Lie group.    
Let $ l\,:=\,\dim  \z(\g) $, and $ r\,:=\,\#[\g,\,\g] $. Let $ [\g,\,\g]\,=\, \g_1 \oplus \cdots \oplus \g_r $, where $ \g_i $'s are simple ideals in $ [\g,\,\g] $.  
Moreover assume that  $ \h \cap [\g,\,\g] \, \subset \,  \g_1\oplus \cdots \oplus \g_s  $ for some $ s\,\leq\, \#[\g,\, \g] $ and  that $ \h\cap \g_i \neq 0$ for all $ i=1,\ldots ,s $. Then 
\begin{enumerate}
		\item $\dim  (H^3(G/H,\, \R)) \,=\,r -s  + \begin{pmatrix}
			l\\
			3
		\end{pmatrix} $,
		\item $ \dim  (H^4(G/H, \,\R)) \,=\, l(r-s) \,+\,  \#(\h,H) -s \,+\,  
		      \begin{pmatrix}
			l\\
			4
		\end{pmatrix} $.
	\end{enumerate}
\end{corollary}

\begin{proof}
Using  Lemma \ref{lem-decomposition-of-S(h)-new}, it follows that 
$ \dim  S^2([\g,\,\g]^*)^G\,=\, r  $ and $ \dim  S^2(\h^*)^H\,=\,\#(\h,\,H)$. In view of Proposition \ref{prop-dimNgh}\eqref{prop-4.4-3},  we have $ \dim  \NC_{\g, \h}= r-s $. Thus $ \dim  \CC_{\g,H}= \#(\h,H) -s$. 
Now proof of the results follows from Theorems \ref{thm-3rd-coh-general}, \ref{thm-4th-coh-general}.
\end{proof}

If $G$ is a general compact connected Lie group and $H$ is a closed subgroup, then the finite group $H/H^0$ has a natural right action on $G/H^0$ turning it into a $H/H^0$ principal bundle over $G/H$.  Thus one has the
following inequality of the Betti numbers in all dimensions:
$ \dim   H^i (G/H , \R) = \dim  H^i (G/H^0 , \R)^{H/H^0} \leq \dim  H^i (G/H^0 , \R)$ for all $i$.  
One may observe, in view of Theorem \ref{thm-3rd-coh-general} and Theorem \ref{thm-4th-coh-general}, that the equality does not hold
even in dimensions three and four; see Example \ref{example} below. 
\begin{example}\label{example}
Let $G\,:=\,{\rm U}(2) \times \mathbb{S}^1$, and $H\,:=\, {\rm O}(2)\times 1 \, \subset\, G$. Then $ \z(\h)
\,=\,\s\o(2)\subset [\g,\,\g]$, and ${\rm O}(2)/{\rm SO}(2)\,=\,\Z/2\Z$ acts non-trivially on $\s\o(2)$. In this case $\dim \NC_{\g, \h}
\,=\, 0\,=\,\dim \CC_{\g,H}$. Using Theorem \ref{thm-3rd-coh-general} and Theorem \ref{thm-4th-coh-general}, we conclude that
$\dim H^3(G/H, \,\R) \,=\,0\, \neq \,    \dim H^3(G/H^0, \,\R) \,=\,2,$ and  $\dim H^4(G/H,\, \R) \,=\,\,0\, \neq \,
\dim H^4(G/H^0,\, \R) \,=\,1.$\qed
\end{example}

\begin{remark}
In this remark, we will illustrate a  natural situation where Corollary \ref{cor-h-gg34}   can be applied.  Let $ G $ be as in Corollary \ref{cor-h-gg34}.  Let $ H\subset G$ be a closed subgroup.  Suppose there is a $s$ with $s \leq r$ such that  $\h$ can be decomposed as $\h = \h_1 \oplus \cdots \oplus \h_s$ where $\h_i,  \, i = 1,  \dots, s$ are semisimple (not necessarily simple) ideals of $\h$ with $0 \neq \h_i \subset \g_i$ for all $ i = 1,  \dots,s$.  Then one may apply Corollary \ref{cor-h-gg34} to obtain the third and fourth Betti numbers of $G/H$.\qed
\end{remark}

The next result follows immediately either  from Theorems \ref{thm-3rd-coh-general},  \ref{thm-4th-coh-general} or  Corollary \ref{cor-h-gg34}.
\begin{corollary}
Let $ G $ be a compact, connected Lie group, $ H\subset G$ a closed subgroup so that $H^0$ is a  simple Lie group.   Let $ l:= 	\dim  \z(\g) $, and $ r\,:=\,\#[\g,\,\g] $. Then
	\begin{enumerate}
		\item $\dim  (H^3(G/H,\, \R)) \,=\,r -1  + \begin{pmatrix}
			l\\
			3
		\end{pmatrix} $,
		\item $ \dim  (H^4(G/H, \,\R)) \,=\, l (r-1) \,+\,  \begin{pmatrix}
			l\\
			4
		\end{pmatrix} $.
	\end{enumerate}
\end{corollary}

{\bf Proof of Corollary \ref{dim-7}.}
Observe that $G/ H^0$ is compact orientable and  $\dim  G/ H^0 \,=\, 7$.  Using Poincar\'{e} duality, we have $\dim  H^3 (G/H^0) \,=\, \dim   H^4 (G/H^0)$. Now we apply Corollary \ref{cor-g-semisimple}(4) to conclude that $\#\g \,=\, \# \h$.
\qed 

\begin{remark}
Let $G$ be a compact connected Lie group and $ H $ be a closed subgroup of $G$ so that $G/H$ is orientable.
Assume further  that $\dim   \,(G/H)\,\leq \, 9$.  Then using  \cite[Theorem 3.6,  Theorem 3.3]{BCM}, Theorem \ref{thm-3rd-coh-general}, Theorem \ref{thm-4th-coh-general}, and  Poincar\'{e} duality one may compute dimensions of all the cohomologies of $G/H$.
\qed
\end{remark}

\subsection{Homogeneous spaces which are not necessarily compact}\label{appl-thm-noncompact}
In this section, we draw some consequences in the setting of homogeneous spaces of Lie groups which are not necessarily compact. We need a result due to Mostow using which computation of the cohomologies of  a general connected homogeneous space  boils down to that of certain compact homogeneous (sub)space.

\begin{theorem}[\cite{Mo}]\label{mostow} 
	Let $L$ be a connected Lie group, and let $M \,\subset\, L$ be a closed subgroup with finitely many connected  components. Let $G$ be a maximal compact subgroup of $L$ such that $G \cap M$ is a maximal compact subgroup of $M$. 
	Then the image of the natural embedding $G/ (G \cap M)\, \hookrightarrow\, L/M$ is a deformation retraction of 
	$L/M$.
\end{theorem} 

Theorem \ref{mostow} is proved in \cite[p. 260, Theorem 3.1]{Mo} under the assumption that $M$ is connected. However, as mentioned in \cite{BC}, using \cite[p. 180, Theorem 3.1]{H}, the proof as in \cite{Mo} goes through when $H$ has finitely many connected components.

Let $L,\, M,G\,$ be as in Theorem \ref{mostow}, and let $H \,:= \,G \cap M$. As
$\ G/H \, \hookrightarrow \, L/M$ is a deformation retraction by Theorem \ref{mostow}, we have 
\begin{equation}\label{homotopy}
	{ H}^{i}(L/M ,\, \R) \,\simeq\, { H}^{i}(G/H ,\, \R ) ~\ \ \text{ for all } \ i\, . 
\end{equation}

In view of Theorem \ref{mostow} we next recast Theorem \ref{thm-3rd-coh-general} and Theorem \ref{thm-4th-coh-general} in the setting of a general 
connected homogeneous space which is not necessarily compact.

\begin{theorem}\label{thm-3rd-coh-general-non-cpt}
Let $L$ be a connected Lie group and $M\,\subset\, L$ a closed subgroup with finitely many connected components. Let $G$ be a maximal compact subgroup of $ L $, and $ H $  a maximal compact subgroup of $ M $ with $H\,\subset\, G$. Then there is an isomorphism: 
	\begin{align*}
		H^3(L/M )\, \simeq\,  \big(({\z}(\h) \cap [\g,\,\g])^*\big)^{H/H^0} \ox {\Big(\frac{\g}{[\g,\,\g]+\h}\Big)}^*
		\,\bigoplus \, \NC_{\g,\h}  \, \bigoplus  \, \wedge^3 {\Big(\frac{\g}{[\g,\,\g]+\h}\Big)}^*\,.
	\end{align*}
\end{theorem}

\begin{theorem}\label{thm-4th-coh-general-non-cpt}
	Let $L,\, M, \,G$ and $H$ be as in Theorem \ref{thm-3rd-coh-general-non-cpt}.
	Let $r_0\,:=\,\dim   \big(\frac{\g}{[\g,\,\g]+\h}\big)$. Then there
	is an isomorphism: 
	\begin{align*}
		H^4(L/M )\,\ \simeq\,\   &  
		\big(({\z}(\h) \cap [\g,\,\g])^*\big)^{H/H^0} \ox\wedge^2 {\Big(\frac{\g}{[\g,\,\g]+\h}\Big)}^* \,
		\bigoplus   \, \,
		(\NC_{\g,\h})^{r_0}  \\
		&\,  \,   \bigoplus\,\,  \CC_{\g,H} 	\,\,\bigoplus  \,\,\wedge^4 {\Big(\frac{\g}{[\g,\,\g]+\h}\Big)}^*\, .
	\end{align*}
\end{theorem}

We will now restrict our attention to certain special types of (non-compact) semisimple Lie groups
to obtain better results.  
We first define a notion associated to a real simple Lie group which will facilitate formulating the next set of corollaries. In addition to 
the foregoing fact regarding the dimension of the center of any maximal compact subgroup, we also need to recall that in a simple Lie group the 
semisimple part of a maximal compact subgroup is not necessarily simple. However, the number of simple factors of the semisimple part of a maximal compact subgroup in a simple Lie group is at the most two; see \cite[Appendix C]{Kn}.

\begin{definition}\label{def-type-I-type-II}
Let $L$ be a (connected) simple Lie group.  
Then {\it $L$ is said to be of Type-I}  if a maximal compact subgroup of $L$ is a simple Lie group, and {\it $L$ is said to be of  Type-II} if 	a maximal compact subgroup of $L$ is a product of two simple Lie groups. See \S \ref{appendix} for a complete list of all the Type-I and Type-II simple Lie groups.
\qed
\end{definition}

\begin{corollary}\label{cor-g-semisimple-noncpt-new}
	Let $L$ be a semisimple group,  and $M$ be 
	a semisimple subgroup such that all the simple
	factors of both $L$ and $M$ are either
	of Type-I or of Type-II.  Assume that 
		$L$ admits exactly $r_1$-many simple factors of Type-I and $s_1$-many simple factors of Type-II. 
	\begin{enumerate}
		\item 
		Then 
		$\dim  H^3(L/M ) \,\leq\, 	 (r_1+2s_1)$.
		\item  Moreover, if  $ M $ admits exactly  $r_2$-many  simple factors of Type-I and   $s_2$-many simple factors of Type-II, then 
		$\dim  H^4(L/M ) \,\leq\, ( r_1+r_2) + 2(s_1+s_2)-1$.
	\end{enumerate}
\end{corollary}

\begin{proof} 
Let $H$ be a maximal compact subgroup of $M$, and $G$ be a maximal compact subgroup of $L$  	containing $H$. 	It is clear from Proposition \ref{prop-max-cpt} that  	$\dim \z(\g) \leq 1, \dim \z(\h) \leq 1$.  Moreover,  we have that  $\# [\g, \,\g]
\,=\, r_1 + 2s_1$, and $\# [\h,\, \h] \,=\, r_2 + 2s_2$.  We now use Proposition 	\ref{prop-dimNgh} (1) to conclude that  	$\dim \NC_{\g,\h} \leq r_1+ 2s_1 -1$, and Lemma \ref{lem-decomposition-of-S(h)-new} (4) to conclude that $ \dim S^2((\h\cap [\g,\,\g])^* )^H\,\leq \, 1+ \#[\h,\h] $. Thus  	$\dim \CC_{\g,H} \leq r_2+ 2s_2$; see Definition \ref{def-NC-gH}.  	We now use Theorem \ref{mostow} and Corollary \ref{cor-h<gg} \eqref{cor-dim-zg<1} to see that 
\begin{align*}
&\dim H^3(L/M ) \leq 1 + \dim \NC_{\g,\h}
\leq r + 2s\,, \, {\rm and }\\
\dim  H^4  &(L/M ) \,\leq\,    \dim \NC_{\g,\h}   \,+\,\dim \CC_{\g,H}   \,\leq\, r_1+r_2+2(s_1+s_2)-1\,.	
\end{align*} 
\end{proof}

We will now derive vanishing results of the third and fourth
 cohomologies of homogeneous space
in certain non-compact setting. In view of Theorem \ref{mostow}, the next result follows either from \cite{A} or from Corollary
\ref{cor-g-semisimple}\eqref{cor1.3-2b}.

\begin{corollary}
Let $L$ be a simple Lie group of Type-I. Let $ M$ be a closed subgroup of $ L $ with finitely many components. Assume that $ M $ contains a compact subgroup of positive dimension. Then $ H^3(L/M,\R) =0$.
\end{corollary}

In the next corollary we generalize the foregoing one.

\begin{corollary}	Let $ L $ be a  semisimple Lie group, and $ M\subset L $ be a closed  semisimple   subgroup.  
Assume that both  $L$ and $ M$  admit  semisimple  maximal compact subgroups. 
Let $ L_1, \dots ,L_{r} $ be the simple factors of $ L $, and each $ L_i $ is of Type-I. 	Suppose that $ L_i \cap M $ contains a compact subgroup of positive dimension for each $ i=1,\dots ,r $. Then 
	$$ H^3(L/M,\R) \, =\,  0 \,.$$  
\end{corollary}
\begin{proof}	Let $ C_i \subset L_i \cap M$ be  a compact subgroup of positive dimension. Then $ C:= C_1\cdots C_r$ is a compact subgroup in $ M$. Let $H$ be  a maximal compact subgroup of  $M $ containing $ C $, and  $ G $ be  a maximal compact subgroup of  $L $ containing $H $. Then   by Corollary \ref{cor-max-cpt},  $ G=G_1\cdots G_r $ where $ G_i  $ is a maximal compact subgroup of $ L_i $ for each $ i=1,\dots ,r $. 	Since $ L_i $ is of Type-I, $ G_i $ is simple for all $ i $. Next we 	claim that  $ \c_i\, \subset \g_i \  \  \forall \,i$. To see this note that $\c_i\,\subset \,  \g\,=\, \g_1+\dots +\g_r $. Let $ X\in \c_i $ and $\c_i \subset \l_i $ . Then $ X=X_1+\cdots +X_r $ for some $ X_i\in \g_i \subset \l_i $. Thus 
	$$ X-X_i \,=\, X_1+\cdots X_{i-1}+X_{i+1}+ \cdots +X_r\,.   $$
But $ X-X_i \, \in \, \l_i $ and  $X_1+\cdots X_{i-1}+X_{i+1}+ \cdots +X_r \, \in \l_1 \oplus\cdots \oplus\l_{i-1} \oplus \l_{i+1} \oplus \l_r $.
Hence $ X-X_i =0$ which proves the claim.  Thus $ \c_i \subset \g_i\cap \h $, and hence  $\dim \g_i\cap \h >0$. Now in view of Theorem \ref{mostow}, the proof follows from Corollary \ref{cor-h-gg34}. 
\end{proof}

\begin{corollary}
Let $L$ be a semisimple  Lie group, and $ M\subset L $ be a closed semisimple subgroup. 
Assume that both  $ L $ and $ M$  admit  semisimple  maximal compact subgroups.  
Let $ L_1, \dots, L_{r} $ be the simple factors of $ L $, and $ M_1,\dots, M_{s} $ be the simple factors of $M$ such that $r\geq s  $. Assume that $M_i \subset L_i$ and both  $L_i,\,M_i$ are of Type-I for all $ i=1,\dots ,s$. Then $$ H^4(L/M,\R) =0 \, .$$  
\end{corollary}

\begin{proof}
Let $H_i\subset M_i$ be a maximal compact subgroup. Since $ M_i $ is of Type-I, $ H_i $ is simple for all $i$. Let  $G_i\subset L_i$ be a maximal compact subgroup containing $ H_i $.  Set $ H:=H_1\cdots H_s \subset M$ and $ G:=G_1\cdots G_r \subset L$. 
Then using Corollary \ref{cor-max-cpt}, it follows that  $ H$ is a maximal compact subgroup of $ M $ and  $G$ is a maximal compact subgroup of $L$ containing $ H $. Here $ \h\subset \g_1\oplus \cdots \oplus \g_s $ and $ \h\cap \g_i \, \neq \, 0$ for all $ i=1, \dots s $. Since $ H $ is connected and semisimple, $ \#(\h, H) = \#\h=s$. Now in view of Theorem \ref{mostow}, the proof follows from Corollary \ref{cor-h-gg34}. 
\end{proof}

\begin{corollary}
Let $ L $ be a connected Lie group which admits a semisimple maximal compact subgroup. Let $ M\subset L$ be a closed subgroup with finitely many connected components, and $ M^0 $ be a simple Lie group of Type-I. Then  $ H^4(L/M,\R) =0$.
\end{corollary}

\begin{proof}
In view of Theorem \ref{mostow}, the proof follows from Corollary \ref{cor-g-semisimple}\eqref{cor1.3-2c}.
\end{proof}

\section{Appendix : Maximal compact subgroups and types of simple Lie groups}\label{appendix}

In this appendix we deal with two points which are required in \S \ref{appl-thm-noncompact}
:  in the first part  we describe certain subtle issues regarding relations between a maximal compact
subgroup of a semisimple groups and those of its simple factors, and in the second part we list all the Type-I, Type-II simple Lie groups.

\begin{lemma}\label{contractible} Let $G$ be a connected Lie group. 
\begin{enumerate}
\item Let $A$ be a connected compact subgroup. Then the coset space $G/A$ is
contractible if and only if $A$ is a maximal compact subgroup.

\item Let $B$ be a connected closed subgroup such that the coset space $G/B$ is
contractible. Then a maximal compact subgroup in $B$ is a maximal compact subgroup in $G$.
\end{enumerate}
\end{lemma}

\begin{proof}
We first prove (1). If $A$ is a maximal compact subgroup, then by \cite[Theorem 3.1., pp. 180--181]{H},  
$G/A$ is contractible.
Now suppose that $G/A$ is contractible. We will show that $A$ is a maximal compact subgroup in $G$.
 As $A$ is connected compact, there is a maximal compact subgroup 
$H$ of $G$ such that $A \subset H$. Then by \cite[Theorem 3.1., pp. 180--181]{H},   $G/H$ is contractible. Observe that $G/A$ is a fiber bundle
over $G/H$ with fiber $H/A$. As the base $G/H$ is contractible  it follows that $G/A \simeq G/H \times H/A$
as fiber bundles. Thus $G/A$ has the same homotopy type as $H/A$, and hence $H/A$ is contractible. 
Recall that a compact manifold without boundary 
can not be contractible if it is of positive dimension
(this is because if $M$ is a compact manifold
without boundary and if 
$n := \dim M$ the $n$-th singular homology 
with coefficient in $\Z_2$ satisfies
$H_n ( M, \Z_2) = \Z_2$).
Thus we have that 
$\dim H/A = 0$ which forces  $A = H$ as $H$ is connected. Hence $A$ is a maximal compact subgroup  in $G$.

We now prove (2). Let $K \subset B$ be a maximal compact subgroup of $B$. 
Note that $G/K$ is a fiber bundle over $G/B$ with fiber $B/K$. As the base $G/B$ is
contractible it follows,  as above,  that $G/K \simeq G/B \times B/K$ as fiber bundles.
As $K$ is a maximal compact subgroup by  \cite[Theorem 3.1., pp. 180--181]{H} $B/K$ is contractible. As $G/B$ is contractible, so is
$G/K$. As $K$ is compact, by (1), it follows that $K$ is a maximal compact subgroup of $G$. 
\end{proof}

The next lemma is formulated in a slightly more general context, 
but  we need only a special case of it in the
proof of Proposition 
\ref{prop-max-cpt}.

\begin{lemma}\label{lemma-max-cpt} Let $C$ be a connected compact Lie group, and let $V$ be a finite dimensional vector space over $\R$.  Let $n:= \dim  V$.
Let  $H:= V \times C$ be the direct 
product of the vector group $V$ and $C$. Let $\Gamma \subset H$ be a discrete central subgroup
of $H$, and let $\theta \,:\, H \,\longrightarrow\, H/\Gamma$ be the quotient homomorphism.
\begin{enumerate}
\item Then $\Gamma$ is a finitely generated abelian group, and 
moreover,  ${\rm rank} \, \Gamma \,\leq\, n$.
\item Suppose ${\rm rank} \, \Gamma \,=\,m$. 
Then there exists closed subgroups $W, T \,\subset\, H/\Gamma$
 such that $W $ is isomorphic to
$\R^{n-m}$,  $T$ is isomorphic to the $m$-dimensional
 compact torus  $(\SS^1)^m$ and
 $H/\Gamma \,=\, W \, T\,  \theta (C)$.
Moreover, any pair of subgroups among the subgroups
$W,\, T,\, \theta (C)$ commute and intersect
in a discrete subgroup. In particular,
$H/\Gamma$ is isomorphic to $W \times (T \theta (C))$, and, consequently, $T \theta (C)$ is 
a maximal compact subgroup of $H/\Gamma$.
\end{enumerate}
\end{lemma}

\begin{proof} We prove (1) first.
As $\Gamma$ is discrete and central it follows that $\Gamma$ is finitely generated. 
Let  $\pi \,:\, H \,=\, V \times C \,\longrightarrow\, V$
denote the projection to the first factor. As $C$ is compact, $\pi$
is a proper map, and hence $\pi (\Gamma)$ is a discrete subgroup of $V$. Thus
${\rm rank} \, \pi(\Gamma) \leq n$. Now it is easy to see that the group of torsion elements
of $\Gamma$ is $ \ker \pi \cap \Gamma$, and hence  ${\rm rank} \, \Gamma
\,=\, {\rm rank} \, \Gamma /( \ker \pi ) \cap \Gamma \,=\, {\rm rank} \, \pi(\Gamma) \,\leq\, n$.

We now prove (2). As  ${\rm rank } \, \Gamma  \,=\, {\rm rank } \, \pi (\Gamma)
\,=\, m$ we may choose elements in $\Gamma$, say,  $(v_1,\, c_1),\, \cdots,\, (v_m,\, c_m)$  such that
$\{v_1,\, \cdots,\, v_m\}$ is a $\Z$-basis of $\pi (\Gamma)$. 
Since $\pi (\Gamma)$ is a discrete subgroup of $V$
it follows that $\{v_1,\, \cdots,\, v_m\}$ is $\R$-linearly independent.
Let $W_1 \,:=\, \R v_1 + \cdots + \R v_m$, and let $W$ be a subspace of $V$ 
complementary to $W_1$. Let $T$ be a maximal torus of the compact Lie group $C$.
As  $c_1,\,  \cdots,\, c_m \,\in \,Z (C)$ it follows that $c_1,
\,\cdots,\, c_m \,\in\, T$.  Choose $X_1,\, \cdots,\, X_m \,\in\, {\rm Lie}(T)$
 such that $\exp (X_i ) \,=\, c_i $ for all $i$.
We now set  a map $\phi \,:\, \R^m \,\lto\, V \times C$ by
$$\phi( (t_1,\, \cdots,\, t_m)) \,:=\, (t_1 v_1 + \cdots + t_m v_m, \, 
\exp (t_1 X_1 + \cdots+ t_m X_m) ),  \,\, (t_1,\, \cdots, \,t_m ) \,\in\, \R^m.
$$

As the Lie algebra  ${\rm Lie}(T)$ is  abelian, the map
$\phi$ is a (smooth) group homomorphism. 
We now show that $\phi ( (t_1,\, \cdots,\, t_m) ) \,\in\, \Gamma $ if and only if 
$ (t_1,\, \cdots,\, t_m) \,\in\, \Z^m$. It is easy to see that $\phi ( \Z^m) \,\subset\, \Gamma$.
Now if $(t_1 v_1 + \cdots + t_m v_m, \, 
\exp (t_1 X_1 + \cdots+ t_m X_m) ) \,\in\, \Gamma$ then
$t_1 v_1 + \cdots + t_m v_m \,\in\, \pi (\Gamma)$. As $\{v_1,\, \cdots,\, v_m\}$ is $\R$-linearly
independent and is a $\Z$-basis for $\pi (\Gamma)$ it follows that $t_1,\, \cdots,\, t_m \,\in\, \Z$. 
As  $\R^m / \Z^m $ is compact one has the embedding of groups, say,  $\phi' \,:\, \R^m / \Z^m 
\,\longrightarrow\, W \times C / \Gamma$ which is induced from $\phi$.
Let $T \,:=\, \phi' ( \R^m / \Z^m)$. Then $T \,\simeq\, (\SS^1)^m$. 
We next show that $T \cap \theta (C)$ is finite. For this one notices that if 
 $(t_1 v_1 + \cdots + t_m v_m, \, 
\exp (t_1 X_1 + \cdots+ t_m X_m) ) \,=\, (0,\, c) \, {\rm mod} \, \Gamma$.
Then $t_1 v_1 + \cdots + t_m v_m
\,\in\, \pi (\Gamma)$. Thus $t_i \,\in\, \Z$ for all $i$.
Thus $(t_1 v_1 + \cdots + t_m v_m, \, 
\exp (t_1 X_1 + \cdots+ t_m X_m) ) \in \Gamma$, and hence $(0,\, c) \,\in\, 0 \times C \cap \Gamma$.
Thus $T \cap \theta (C) \,\subset\, \theta ( C \cap \Gamma)$,  and hence $T \cap \theta (C)$ is finite.
\end{proof}

In this section we also require facts associated to the Cartan decomposition in a semisimple Lie group  for which we refer to \cite[p.~362, Theorem 6.31]{Kn} for details.  If $L$ is a semisimple Lie group and if
$\sigma \,:\, \widetilde{L}\,\longrightarrow\, L$ is a covering homomorphism 
then we may choose Cartan involutions 
$\Theta \,:\, L \,\longrightarrow\, L,\,\, \widetilde{\Theta} \,: \,\widetilde{L}
\,\longrightarrow\,
\widetilde{L}$, aligning with each other, that is, 
$\Theta \circ \sigma \,=\, \sigma \circ \widetilde{\Theta}$.
If $G$ is a semisimple group
and $K$ denotes the group of  fixed points
of a Cartan involution then $Z(G) \subset K$.
Moreover, it follows from \cite[p.~362, Theorem 6.31]{Kn} that
 $K/Z(G)$ is a maximal compact subgroup of $G/Z(G)$. 
Thus, one may observe from the facts above that $K$ need not be a maximal compact 
subgroup in $G$ but any maximal compact subgroup of $K$ is a maximal compact subgroup of $G$.
However, the Lie algebra $\k$ is a compact Lie algebra (hence it is reductive).
We conclude that 
 any maximal compact subgroup of a simply connected semisimple group is either semisimple
 or the trivial group.
 We also need Weyl's theorem which says that the universal cover of a compact semisimple group is compact; see
\cite[p.~268, Theorem 4.69]{Kn}. 

If $L$ is a connected Lie group and if $L$ is 
an almost direct product of the distinct closed connected subgroups
$L_1,\, \cdots,\, L_n$ then $L_i $ and $L_j$ commute if $i \,\neq\, j$.   Moreover,
$Z (L) \,=\, Z(L_1) \cdots Z(L_n)$.  The kernel of the homomorphism
$\phi \,:\, L_1 \times \cdots \times L_n \,\longrightarrow\, L$
defined by
$\phi (x_1, \, \cdots,\, x_n) \,:=\, x_1 \cdots x_n$ is isomorphic to $L_1 \cap L_2 \cdots L_n
\,\simeq\, L_i \cap L_1 \cdots \widehat{L_i} \cdots L_n$ (here $\widehat{A}$ means $A$ is omitted).
Moreover,  for all $i$,
$
L_i \cap L_1 \cdots \widehat{L_i} \cdots L_n
\,=\, Z(L_1) \cap Z(L_1) \cdots \widehat{ Z(L_i)} \cdots Z(L_n).
$

\begin{proposition}\label{prop-max-cpt}
Let $ L $ be a semisimple Lie group. Suppose  $ L_1,\, \cdots ,\,L_{r} $
 are all the simple factors of $ L $. Let $\Gamma \,:=\, L_1\cap L_2\cdots L_r $.
Then the following statements hold:
\begin{enumerate}

\item If  $ |\Gamma | \,<\, \infty $,  then  for any maximal compact subgroup $ K $ of $ L $
there exist maximal compact subgroup $K_i$ in $ L_i $ for all $ i $ such that $K\,=\,
K_1\cdots K_r $.

\item If  $ |\Gamma | \,=\, \infty $, then  $\Gamma$
is a finitely generated (abelian) group of rank $ 1 $.  One also has that,
for all $i$,  if $K_i$ is a the subgroup of fixed points
corresponding to a (global) Cartan involution of $L_i$ then
$K_i \,\simeq\, \R \times [K_i,\, K_i]$ and  $[K_i,\, K_i]$
is a maximal compact  subgroup of $L_i$.  Moreover, for all $i$,
either $K_i$ is abelian (this happens only when
$L_i \,\simeq\, \widetilde{ {\rm SL}_2 (\R)}$, the simply connected cover of ${\rm SL}_2 (\R)$) or 
$[K_i, K_i]$ is compact semisimple.
Further, for
any maximal compact subgroup $ K $ of $ L $ 
there exist Cartan involutions of $L_i$, for all $i$ such that if $K_i$ is the corresponding  subgroup of fixed points then $K= 
 T [K_1,\, K_1] \cdots [K_r,\, K_r] $ where $T$ is a central (compact) torus of $K$ with $\dim T \,=\,1$.
 In particular,  any maximal compact subgroup
 of $L$ is not semisimple and the center is one dimensional.
 \end{enumerate}
\end{proposition}

\begin{proof} We begin by claiming that ${\rm rank} \, Z(L_1) \cap Z(L_2) \cdots Z(L_n)
\,\leq\, 1$. Recall that $Z(L_1) \cap Z(L_2) \cdots Z(L_n)$ 
is a discrete finitely generated abelian subgroup  of $Z(L)$.
Let $K_1$ be the fixed point subgroup of a Cartan involution of $L_1$.  Then  
 $Z(L_1) \subset K_1$; see \cite[p.~362, Theorem 6.31]{Kn}.  Thus $Z(L_1) \,
\subset\, Z(K_1)$,  and it is well-known that $\dim Z(K_1) \,\leq\, 1$.  Thus either $\dim Z(K_1) \,=\,0$ or 
$Z(K_1) \simeq \SS^1$ or $Z(K_1) \simeq \R$.
Since $Z(L_1) \cap Z(L_2) \cdots Z(L_n)$ is
a discrete subgroup the claim follows.

For the proof we set the homomorphism
$\phi \,:\, L_1 \times \cdots \times L_n \,\longrightarrow\, L$
defined by
$\phi (x_1,  \dots, x_n) \,:=\, x_1 \cdots x_n$. Let $\Gamma \,:=\, \ker \phi$. Then $\Gamma
\,\simeq\, L_1 \cap L_2 \cdots L_n$.

If $ |L_1 \cap L_2 \cdots L_n| \,<\, \infty$ then
the homomorphism $\phi$ is a finite
covering map (hence proper).  Let $K \,\subset\,
L$ be a maximal compact then $\phi^{-1} (K)$ is a compact subgroup of $L_1 \times \cdots \times L_n$.  Let
$M$ be a maximal compact of 
$L_1 \times \cdots \times L_n$ such that $M \subset \phi^{-1} (K)$.  It is easy to see that 
there exists
maximal compact subgroups $K_i $ of $L_i$ for all
$i$ such that $M \,=\, K_1 \times \cdots \times K_n$.
Now as $K \,\subset \,K_1 \cdots K_n$ by maximality
it follows that $K \,=\, K_1 \cdots K_n$.

Let $ |L_1 \cap L_2 \cdots L_n| \,=\, \infty$.
Let $K_i $ be the fixed point subgroup of a Cartan
involution of $L_i$, for all $i$. Recall that $Z(L_i) \,\subset\, K_i$ for all $i$.
It
follows that $|L_i \cap L_1 \cdots \widehat{L_i} \cdots L_n| \,=\, \infty$ for all $i$.  As 
$L_i \cap L_1 \cdots \widehat{L_i} \cdots L_n\,=\,
Z(L_i) \cap Z(L_1) \cdots \widehat{Z(L_i)} \cdots Z(L_n)$ it follows that $Z(K_i) \,\simeq\, \R$ for all $i$.
Observe that $\Gamma \subset Z(L_1) \times 
\cdots \times Z(L_n) \,\subset \,
Z(K_1) \times 
\cdots \times Z(K_n)$.
Since $L_1 \times \cdots \times L_n/ K_1 \times \cdots \times K_n$ is contractible, by Lemma \ref{contractible} (2), 
a maximal compact of $K_1 \times \cdots \times K_n / \Gamma$ is a maximal compact of 
$L_1 \times \cdots \times L_n / \Gamma$
(note that $K_1 \times \cdots \times K_n / \Gamma$
naturally embeds in $L_1 \times \cdots \times L_n / \Gamma$).
Now recall by Weyl's theorem that $[K_i, K_i]$ is
either a trivial group or a compact semisimple group.  As ${\rm rank } \, \Gamma \,=\, 1$
and 
$K_1 \times \cdots \times K_n  \simeq \R^n \times C$ where $C$ is a compact group, by Lemma
\ref{lemma-max-cpt},  the results follows. 
\end{proof}

The next result follows immediately from Proposition \ref{prop-max-cpt}.

\begin{corollary}\label{cor-max-cpt}
  Let $ L $ be a semisimple Lie group such that a maximal compact subgroup of $ L $ is semisimple.  Let
$ L_1,\, \cdots ,\,L_{r} $ be the simple factors of $ L $. Then any maximal compact subgroup
of $ L $ is of the form $ G_1\cdots G_r $ where $ G_i  $ is a maximal compact subgroup of $ L_i $
for each $ i\,=\,1,\,\cdots ,\,r $.   
\end{corollary}

\begin{remark} We remark here that semisimplicity of 
maximal compact subgroups in each simple factor 
of a semisimple group $L$ does not force semisimplicity
of a maximal compact subgroup in $L$.  
One can construct an example as in the following way.
Let $\widetilde{{\rm SU} (p,q)}$ be the universal 
cover of ${\rm SU} (p,q)$.  Then 
one can see that the fixed point subgroup of any
global Cartan involution in $\widetilde{{\rm SU} (p,q)}$ is isomorphic to $\R \times {\rm SU} (p) \times {\rm SU} (q)$. Thus any maximal compact subgroup in 
$\widetilde{{\rm SU} (p,q)}$ is isomorphic to 
${\rm SU} (p) \times {\rm SU} (q)$, and hence it is
semisimple.  It is clear that there is a discrete central subgroup $\Lambda \subset \widetilde{{\rm SU} (p,q)}$ such that $\Lambda \simeq \Z$.  Let
$\Gamma \,:=\, \{ (\gamma, \gamma) \,\big\vert\,\, \gamma \,\in\,
\Lambda \}$.
Now 
consider $L\,:=\, \widetilde{{\rm SU} (p,q)} \times
\widetilde{{\rm SU} (p,q)}/ \Gamma$, and let $\theta
\,:\, \widetilde{{\rm SU} (p,q)} \times
\widetilde{{\rm SU} (p,q)} \,\longrightarrow\, L$ be the quotient homomorphism.  Let $I \in {\rm SU} (p,q)$ be the identity element,  
$L_1 \,:=\, \theta 
(\widetilde{{\rm SU} (p,q)} \times \{I\})$, and $L_2
\,:= \,\theta 
(\{I\} \times \widetilde{{\rm SU} (p,q)} )$.  Then
$L_1 \simeq \widetilde{{\rm SU} (p,q)} \simeq L_2$, and $L_1, L_2$ are the simple factors of $L$.
Thus both the simple factors of $L$ have semisimple maximal compact subgroups, but by 
Proposition \ref{prop-max-cpt} (2),  maximal compact subgroups of $L$
are not semisimple.
\qed
\end{remark}

Our next goal is to list all the simple real groups according to their types; see Definition \ref{def-type-I-type-II}. 
We first make the following observations.
Let $G$ be a simple Lie group and $K$ be the group of fixed points of a Cartan involution.
When $\dim   Z(K)^0 \,=\,1$
there are only two possibilities, namely, $Z(K)^0 \,\simeq \,\SS^1$ or
$Z(K)^0 \,\simeq\, \R$. We list below
various equivalent conditions for these possibilities to hold,
which are easily verifiable.

The following conditions are equivalent to  $Z(K)^0 \,\simeq\, \R$:
\begin{align*}
	&   \bullet\, K \,\simeq \,\R \times [K,\, K]   \,,\, \ \bullet\,   |\pi_1 (G,\, e)| \,<\, \infty\,, \ \    \bullet\, | Z (G) | \,=\, \infty,\\
	&  \bullet\,    K \text{  is non-compact}, \, \, \bullet\,   [K,\, K] \text{ is a maximal compact subgroup in}\,   G       .  
\end{align*}
Similarly, the following is a bunch of equivalent conditions for the situation $Z(K)^0 \simeq \SS^1$:
\begin{align*}
	& \bullet\    
	\SS^1 \times [K,\, K] \text{  is a (finite) cover of }  K \, , \  \ \bullet\
	|\pi_1 (G,\, e)| \,=\, \infty  \,,\ \  \bullet\,  \  
	| Z (G) | \,<\, \infty   \,,\\&    \bullet\,
	K \text{ is compact}, \,  \bullet\ 
	K \text{ is a maximal compact subgroup in } G \,,\,  \,   \bullet\ 
	\Z \,\subset\, \pi_1 (G,\, e) .	
\end{align*}

In view of the foregoing observations we list below the simple real groups with respect to their types.
We follow the notation of real simple  Lie algebras as  in \cite[Chapter VI, \S 10]{Kn} and the data
associated with them  as in \cite[Appendix C]{Kn}.
Let us now assume that $G$ is a simple Lie group. Let $\k$ be the 
Lie algebra of $K$. If $\k$ is semisimple then $G$ is either of Type-I or of Type-II (and in this case $K$ is compact).

When $\k$ is simple we list below 
the Type-I and the Type-II simple Lie groups. 
\begin{enumerate}
\item 
The simple Lie group $G$ is of Type-I if and only if 
$G$ is a compact simple Lie group or $G$ is a
complex simple Lie group or $\g$ is
one of the Lie algebras  $ \s\l(n,\R)$  for  $n>2,\, {\s\l}(n,\H), \, {\rm E\, I},\, {\rm E \,IV},\,{\rm E \,V},\, {\rm E \,VIII},\, {\rm F \,II}$.

\item 
The simple Lie group $G$ is of Type-II if and only if 
 the Lie algebra $\g$ is isomorphic to  one of the Lie algebras  $ {\s\p}(p,q)$, ${\s\o}(p,q)$ for $ p \geq q >2 $, $ {\rm E \,II},\, {\rm E \,VI},\,{\rm E \,IX},\,{\rm F \,I},\,{\rm G}$.
\end{enumerate}

We now assume $\dim   \z(\k) =1$.  Then we have the following:
\begin{enumerate}
\item If $\g \simeq \s\l_2 (\R)$ then $G$ is neither of 
Type-I nor of Type-II. In this case a maximal compact subgroup is 
either isomorphic to $\SS^1$ or the trivial group according as $|Z(G)| < \infty $ or $|Z(G)| = \infty $.

\item Let $ \g$ be isomorphic to one of the Lie algebras
$ \s\p(n,\R),\, \s\o^*(2n)$,  $\s\o(n,2)$ for $ n>2 $, ${\rm E\, III}$, $\, {\rm E\, VII}$.
In this case $[\k, \,\k]$ is simple. Thus $G$ is of Type-I
if and only if $|Z(G)| \,=\, \infty$, and the subgroup $K$ is 
a maximal compact with
a one dimensional center if and only if  $|Z(G)| \,<\, \infty$.

\item Let $\g \,\simeq\, \s\u(p,q)$. In this case 
$[\k, \,\k]$ has two simple factors. Thus 
$G$ is of Type-II
if and only if $|Z(G)| \,=\, \infty$, and the subgroup $K$ is a maximal compact with
a one dimensional center if and only if  $|Z(G)| \,<\, \infty$.
 \end{enumerate}

\subsection*{Acknowledgements}
Biswas is supported by a J. C. Bose Fellowship: JBR/2023/000003. Chatterjee acknowledges support from the SERB-DST MATRICS project: MTR/2020/\\000007. 
The authors would like to thank Prof. J. Lauret for bringing \cite{LW} to their attention. They also thank   to the organizers of the  Symposium on Geometry and Topology, held during the {\it 87$^{ th}$ Annual Conference of the Indian Mathematical Society}   on December 5, 2021, and the online Mathematics conference \href{https://sites.google.com/view/world-of-groupcraft-2}{\it World of Group Craft -II} on September 2, 2022, where the some of us presented the main results of this article.

\end{document}